\documentclass[a4paper,11pt]{article}

\usepackage{a4wide}
\usepackage{amsfonts}
\usepackage{amsmath}
\usepackage{amssymb}
\usepackage{amsthm}
\usepackage{graphicx}
\usepackage{latexsym}
\usepackage{graphicx, float, url}

\input amssym.def
\input amssym.tex

\long\def\delete#1{}

\def\Ga{\Gamma}

\def\Om{\Omega}

\def\a{\alpha}
\def\b{\beta}
\def\d{\delta}

\def\s{\sigma}

\def\ve{\varepsilon}

\def\om{\omega}

\def\la{\langle}
\def\ra{\rangle}

\def\mod{{\rm mod~}}

\newcommand{\Z}{\mathbb{Z}}

\newcommand{\Aut}[1]{\mathrm{Aut}(#1)}
\newcommand{\Legendre}[2] {\left( \frac{#1}{#2} \right)}

\def\Cay{{\rm Cay}}

\def\diam{{\rm diam}}

\def\qed{\hfill$\Box$\vspace{12pt}}

\newcommand{\bea}{\begin{eqnarray}}
\newcommand{\eea}{\end{eqnarray}}
\def\non{\nonumber}

\newtheorem{thm}{Theorem}
\newtheorem{cor}[thm]{Corollary}
\newtheorem{lem}[thm]{Lemma}

\newtheorem{remark}{Remark}

\newtheorem{algorithm}{Algorithm}

\theoremstyle{definition}

\newtheorem{ex}{Example}

\title{Frobenius circulant graphs of valency six, Eisenstein-Jacobi networks, and hexagonal meshes}
\author{Alison Thomson and
Sanming Zhou\thanks{Email: smzhou@ms.unimelb.edu.au}\\ \\
Department of Mathematics and Statistics\\
The University of Melbourne\\
Parkville, VIC 3010, Australia}

\date{\today}

\begin{document}
\openup 0.5\jot 
\maketitle

\maketitle

\begin{abstract}
\small{
A Frobenius group is a transitive but not regular permutation group such that only the identity element can fix two points. A finite Frobenius group can be expressed as $G = K \rtimes H$ with $K$ a nilpotent normal subgroup. A first-kind $G$-Frobenius graph is a Cayley graph on $K$ with connection set $S$ an $H$-orbit on $K$ generating $K$, where $H$ is of even order or $S$ consists of involutions. We classify all 6-valent first-kind Frobenius circulant graphs such that the underlying kernel $K$ is cyclic. We give optimal gossiping and routing algorithms for such a circulant and compute its forwarding indices, Wiener indices and minimum gossip time. We also prove that its broadcasting time is equal to its diameter plus two or three. We prove that all 6-valent first-kind Frobenius circulants with cyclic kernels are Eisenstein-Jacobi graphs, the latter being Cayley graphs on quotient rings of the ring of Eisenstein-Jacobi integers.
We also prove that larger Eisenstein-Jacobi graphs can be constructed from smaller ones as topological covers, and a similar result holds for 6-valent first-kind Frobenius circulants. As a corollary any Eisenstein-Jacobi graph with order congruent to 1 modulo 6 and underlying Eisenstein-Jacobi integer not an associate of a real integer, is a cover of a 6-valent first-kind Frobenius circulant. A distributed real-time computing architecture known as HARTS or hexagonal mesh is a special 6-valent first-kind Frobenius circulant. 


}
\end{abstract}

\section{Introduction}

\textsf{A. Introduction.}~~Searching for `good' graphs as models for interconnection networks is an ongoing endeavor in theoretical computer science and network design. It is generally believed that Cayley graphs are suitable network structures due to their many attractive properties (see e.g.~\cite{AK, ABR, CF, Hey, LJD}). In fact, a number of important network topologies \cite{Hey, LJD} such as rings, hypercubes, cube-connected graphs, multi-loop networks, butterfly graphs, Kn\"{o}del graphs \cite{FR}, etc. are Cayley graphs. 

Since the class of Cayley graphs is huge, one may naturally ask which Cayley graphs we should choose in order to achieve high performance. Of course, the answer to this question depends on how we measure the performance of a network, e.g. small diameter, small degree (valency), high connectivity, efficient data transmission, and so on. It has been proved that, as far as routing and gossiping are concerned, a large class of arc-transitive Cayley graphs, called first-kind Frobenius graphs, are `perfect' in the sense that they achieve the smallest possible forwarding indices \cite{FLP, S, Z} and gossiping time \cite{Z}, and possess several other attractive routing and gossiping properties \cite{Z}. (The reader is referred to \cite{BKP, FA, HMS, HKMP} and \S \ref{sec:rg}-\ref{sec:br} for definitions on routing, gossiping and broadcasting.) Because of this and the importance of circulant graphs in network design \cite{BCH, Hwang}, it would be desirable \cite{TZ} to classify all Frobenius circulant graphs and study their behaviours in information communication. In this paper we will classify all $6$-valent first-kind Frobenius circulants such that the kernels of the underlying Frobenius groups are cyclic, and study gossiping, routing and broadcasting in such graphs. We will also study the related family of Eisenstein-Jacobi graphs \cite{FB, MBG} and reveal their connections with 6-valent first-kind Frobenius circulants. 

The reader is referred to \cite{TZ} for a classification of $4$-valent first-kind Frobenius circulants and \cite{FZ} for a recent study on second-kind Frobenius graphs. 

\medskip
\textsf{B. Cayley graphs and circulants.}~~Given a group $K$ and an inverse-closed subset $S$ of $K \setminus \{1\}$ (where  $1$ is the identity element of $K$), the {\em Cayley graph} $\Cay(K, S)$ on $K$ with respect to the {\em connection set} $S$ is defined to have vertex set $K$ such that $x, y \in K$ are adjacent if and only if $xy^{-1} \in S$. A {\em complete rotation} \cite{BKP, Hey, HMP} of $\Cay(K, S)$ is an automorphism of $K$ which fixes $S$ setwise and induces a cyclic permutation on $S$; $\Cay(K, S)$ is {\em rotational} if it admits a complete rotation. 

A Cayley graph on a cyclic group is called a {\em circulant graph} or a {\em circulant}. In computer science, circulants are also called multi-loop networks \cite{BCH, Hwang}. Let $n \geq 7$ and $a,b,c$ be integers such that $1 \le a,b,c \le n-1$ and $a, b, c, n-a, n-b, n-c$ are pairwise distinct. \delete{(Replacing $a, b,c$ by $n-a,n-b,n-c$ when necessary, we may assume $1 \le a,b,c < n/2$.)} Then 
$$
TL_n(a,b,c) = \Cay(\Z_n, S),\;\,S = \{\pm [a], \pm [b], \pm [c]\}
$$ 
is a {\em 6-valent} circulant, that is, every vertex has degree 6. If $a'+b'+c' \equiv 0~\mod n$ for some $a' \in \{a, n-a\}, b' \in \{b, n-b\}, c' \in \{c, n-c\}$, then $TL_n(a,b,c)$ is called {\em geometric}. It is so called because in this case $TL_n(a,b,c)$ can be represented \cite{YFMA} by a plane tessellation of hexagons.  

In this paper we always assume that $TL_n(a,b,c)$ is connected (that is, $\gcd(a, b, c, n) = 1$) and so contains a Hamilton cycle. (In fact, any Cayley graph on an Abelian group with more than two vertices is Hamiltonian; see \cite{CG}.) By relabelling the vertices along a Hamilton cycle, we see that $TL_n(a,b,c)$ is isomorphic to some $TL(a', b', 1)$. Thus without loss of generality we may always assume $c=1$ in $TL_n(a,b,c)$.

\medskip 
\textsf{C. Transitive groups and Frobenius graphs.}~~A group $G$ is said to {\em act} on a set $\Om$ if every $(\a, g) \in \Om \times G$ corresponds to some $\a^g \in \Om$ such that $\a^1 = \a$ and $(\a^g)^{g'} = \a^{gg'}$, where $1$ is the identity element of $G$. The {\em stabiliser} of $\a \in \Om$ in $G$ is the subgroup $G_{\a} = \{g \in G: \a^g = \a\}$ of $G$. The {\em $G$-orbit} containing $\a$ is defined to be $\a^G = \{\a^g: g \in G\}$. If $G_{\a} = \{1\}$ for all $\a \in \Om$, then $G$ is called \cite{Dixon-Mortimer} {\em semiregular} on $\Om$. If $\a^G = \Om$ for some (and hence all) $\a \in \Om$, then $G$ is {\em transitive} on $\Om$. If $G$ is both transitive and semiregular on $\Om$, then it is said to be {\em regular} on $\Om$. 

A transitive group $G$ on $\Om$ is called a {\em Frobenius group} \cite{Dixon-Mortimer} if it is not regular but only the identity element can fix two points of $\Om$. It is well known (see e.g.~\cite[Section 3.4]{Dixon-Mortimer}) that a finite Frobenius group $G$ has a nilpotent normal subgroup $K$, called the {\em Frobenius kernel} of $G$, which is regular on $\Om$. Hence $G = K \rtimes H$ (semidirect product of $K$ by $H$), where $H$ is the stabiliser of a point of $\Om$ and is called a {\em Frobenius complement} of $K$ in $G$. Since $K$ is regular on $\Om$, we may identify $\Om$ with $K$ in such a way that $K$ acts on itself by right multiplication, and we choose $H$ to be the stabiliser of $1$ so that $H$ acts on $K$ by conjugation. For $x \in K$, let $x^H = \{h^{-1}xh: h \in H\}$. A {\em $G$-Frobenius graph} \cite{FLP} is a Cayley graph $\Ga=\Cay(K, S)$ on $K$, where for some $a \in K$ satisfying $\la a^H \ra = K$, $S = a^H$ if $|H|$ is even or $a$ is an involution, and $S = a^H \cup (a^{-1})^H$ otherwise. In these two cases, we call $\Ga$ a {\em first} or {\em second-kind} \cite{Z} Frobenius graph respectively. Since $\la a^H \ra = K$, $\Ga$ is connected in both cases.  

The reader is referred to \cite{Dixon-Mortimer} and \cite{IR, NZ} respectively for group-  and number-theoretic terminology used in this paper. 

\medskip
\textsf{D. Convention.}~~It is possible for a circulant graph to be a Cayley graph on two non-isomorphic groups \cite{M}. Thus it may happen that for a first-kind Frobenius circulant the underlying Frobenius kernel is not isomorphic to a cyclic group. {\em We will only consider first-kind Frobenius circulants with cyclic underlying Frobenius kernels, but for brevity we may not mention this condition explicitly.} 

\medskip
\textsf{E. Main results.}~~The following is a summary of our main results.
\begin{itemize}
\item In \S \ref{sec:frob} we will prove (Theorem \ref{existence}) that there exists a $6$-valent first-kind Frobenius circulant of order $n$ if and only if $n \equiv 1~\mod 6$ and the congruence equation $x^2 - x + 1 \equiv 0~\mod n$ is solvable. Moreover, under these conditions there are precisely $2^{l-1}$ pairwise non-isomorphic such circulants, all of which are arc-transitive and can be constructed from solutions to this congruence equation, where $l$ is the number of distinct prime factors of $n$. 

\item In \S \ref{sec:EJ} we will prove (Theorem \ref{thm:EJ}) that 6-valent first-kind Frobenius circulants are exactly Eisenstein-Jacobi graphs (EJ graphs for short) $EJ_{a+b\rho}$ with $\gcd(a, b) = 1$ whose order is congruent to 1 modulo 6, where $EJ_{a+b\rho}$ is defined \cite{FB, MBG} (see \S \ref{sec:EJ}) to be the Cayley graph on the additive group of $\Z[\rho]/(a+b\rho)$ with respect to the connection set $\{\pm 1/(a+b\rho), \pm \rho/(a+b\rho), \pm \rho^2/(a+b\rho)\}$, where $\rho = (1+\sqrt{-3})/2$ and $\Z[\rho]$ is the ring of Eisenstein-Jacobi integers. We will also prove (Theorem \ref{thm:arc tran}) that all EJ graphs are arc-transitive.

\item We will further prove (Corollary \ref{cor:cover}) in \S \ref{sec:cover} that any EJ graph $EJ_{a+b\rho}$ with order congruent to 1 modulo $6$ and $a+b\rho$ not an associate of any real integer,  is a topological cover of a 6-valent first-kind Frobenius circulant. In fact, we will prove a more general result (Theorem \ref{thm:cover}) which asserts that larger EJ graphs can be constructed from and are covers of smaller EJ graphs. We will also prove a similar result (Theorem \ref{thm:quo}) for the family of 6-valent first-kind Frobenius circulants.

\item The importance of the subfamily of 6-valent first-kind Frobenius circulants lies in that they possess very attractive routing, gossiping and broadcasting properties which are not known to hold for other EJ graphs. These will be discussed in \S \ref{sec:rg} and \S \ref{sec:br}. In \S \ref{sec:rg} we will give optimal gossiping and routing algorithms and compute the forwarding indices and minimum gossip time for any $6$-valent first-kind Frobenius circulant (Corollary \ref{cor:rt-gs}). A by-product is a formula for the Wiener index of any $6$-valent first-kind Frobenius circulant.

\item In \S \ref{sec:br} we will prove that the broadcasting time of any 6-valent  first-kind Frobenius circulant is equal to its diameter plus 2 or 3 (Theorem \ref{thm:broad}), indicating that it is efficient for broadcasting as well.  
\end{itemize}

\textsf{F. Remarks.}~~Three research groups in three different areas came up with the same or related families of graphs independently. The motivation of the present paper, as well as that of \cite{Z} and \cite{TZ}, is to construct Cayley graphs that enable very efficient information transmission. Motivated by construction of perfect codes, in \cite{MBG} Martinez, Beivide and Gabidulin introduced EJ networks; and in \cite{FB} Flahive and Bose further studied EJ networks and related Gaussian networks \cite{MBSMG}. As mentioned above, 6-valent first-kind Frobenius circulants are precisely EJ graphs $EJ_{a+b\rho}$ with $\gcd(a, b) = 1$ and order congruent to 1 modulo 6. This connection enables us to compute distance distributions of 6-valent first-kind Frobenius circulants by using the corresponding results \cite{FB} for EJ graphs. On the other hand, any $EJ_{a+b\rho}$ with order congruent to 1 modulo $6$ and $a+b\rho$ not an associate of any real integer can be constructed from a 6-valent first-kind Frobenius circulant as a topological cover. This allows us to compute the forwarding indices and the minimum gossip times of such EJ graphs by using the general theory developed in \cite{Z}.

The first version of this paper was made public in May 2012 (see \url{http://arxiv.org/pdf/1205.5877v1.pdf}). It was only recently that we found that a very special subfamily of 6-valent first-kind Frobenius circulants had physically been used \cite{CSK, DRS, Shin} as multiprocessor interconnection networks at the Real-Time Computing Laboratory, The University of Michigan. They are called HARTS ({\em Hexagonal Architecture for Real-Time Systems}) \cite{DRS, Shin}, {\em C-wrapped hexagonal meshes} \cite{DRS}, or {\em hexagonal mesh interconnection networks} \cite{ABF}. As we will see in Example \ref{optimal}, they are indeed 6-valent first-kind Frobenius circulants. In the combinatorial community they were first studied in \cite{YFMA}, and their optimal routing and gossiping algorithms were given in \cite{TZ1}.

We remark that the routing problem considered in this paper is different from that studied in \cite{ABF, CSK, FB, MBSMG}, where routing is mainly about computing shortest paths and distance between two vertices. An optimal one-to-all communication (broadcasting) algorithm for HARTS $H_k$ of diameter $k-1$ was given in \cite[Algorithm A2]{CSK}, using $k+2$ steps when $k \ge 3$. Since $H_k$ is a special 6-valent first-kind Frobenius circulant, Theorem \ref{thm:broad} in the present paper can be viewed as a generalization of this result to a much larger family of graphs. In \cite{ABF} an algorithm for all-to-all communication (that is, gossiping in the present paper and \cite{TZ1}) for $H_k$ was given which requires $3k(k-1)/2$ time steps. However, by \cite[Theorem 4]{TZ1}, $k(k-1)/2$ time steps are sufficient and necessary for $H_k$, and an algorithm using $k(k-1)/2$ steps was given in \cite[Algorithm 2]{TZ1}. Yet this is a special case of a more general result: In Algorithm \ref{alg:go} we will give an optimal all-to-all communication algorithm for any 6-valent first-kind Frobenius circulant of order $n$ using $(n-1)/6$ time steps.

\section{Classification of 6-valent  first-kind Frobenius circulants}
\label{sec:frob}

\textsf{A. Preparations.}~~We always use $[m]$ to denote the residue class modulo $n$, where $n$ is a positive integer. Let $\Z_n^* = \{[m]: 1 \le m \le n-1, \gcd(m,n) = 1\}$ be the multiplicative group of units of ring $\Z_n$. We use $[m]^{-1}$ to denote the inverse element of $[m]$ in $\Z^*_{n}$.  It is well known that $\Aut{\Z_n} \cong \Z_n^*$. As the automorphism group of $\Z_n$, $\Z^*_{n}$ acts on $\Z_{n}$ by usual multiplication: $[x][m] = [xm]$, $[m] \in \Z^*_{n}$, $[x] \in \Z_{n}$. $\Z_{n} \rtimes \Z^*_{n}$ acts on $\Z_{n}$ such that $[x]^{([y], [m])} = [(x+y)m]$ for $[x], [y] \in \Z_{n}$ and $[m] \in \Z^*_{n}$. 

\delete
{
The operation of $\Z_{n} \rtimes \Z^*_{n}$ is defined by $([x_1], [m_1])([x_2], [m_2]) = ([x_1]+[x_2][m_1]^{-1}, [m_1 m_2])$ for $([x_1], [m_1]), ([x_2], [m_2]) \in \Z_{n} \rtimes \Z^*_{n}$. Thus the inverse element of $([x], [m])$ in $\Z_{n} \rtimes \Z^*_{n}$ is $(-[xm], [m]^{-1})$.
}

\begin{lem} 
\label{semiregular}
(\cite[Lemma 4]{TZ})
A subgroup $H$ of $\Z_n^*$ is semiregular on $\Z_n \setminus \{[0]\}$ if and only if $[h-1] \in \Z_n^*$ for all $[h] \in H \setminus \{[1]\}$.
\end{lem}

For an odd prime $p$ and an integer $r$, the Legendre symbol $\Legendre{r}{p}$ is defined \cite{IR, NZ} to be 1 if $r$ is a quadratic residue modulo $p$, $-1$ if $r$ is a quadratic nonresidue modulo $p$, and $0$ if $p$ divides $r$. 

An {\em arc} of a graph is an ordered pair of adjacent vertices. A graph $\Ga$ is {\em $G$-arc-transitive} if $G \le \Aut{\Ga}$ is transitive on the set of arcs of $\Ga$, where $\Aut{\Ga}$ is the automorphism group of $\Ga$. 

A group $K$ is a {\em connected $m$-CI-group} \cite{Li} if for any connected Cayley graphs $\Cay(K, S)$ and $\Cay(K, T)$ such that $|S| = |T| \le m$ and $\Cay(K, S) \cong \Cay(K, T)$, there exists $\s \in \Aut K$ such that $T = S^{\s}$. 

\medskip
\textsf{B. Classification.}~~The following is the main result in this section.
\begin{thm}
\label{existence}
Let $n \ge 7$ be an integer. The following statements are equivalent:
\begin{itemize}
\item[\rm (a)] there exists a 6-valent first-kind Frobenius circulant $TL_n(a,b,1)$ of order $n$ such that the kernel of the underlying Frobenius group is cyclic;
\item[\rm (b)] $n \equiv 1~\mod~6$ and the following congruence equation has a solution:
\begin{equation}
\label{eq:3rd}
x^2 - x + 1 \equiv 0~~\mod n.
\end{equation} 
\end{itemize}
Moreover, if one of these conditions is satisfied, then 
\begin{itemize}
\item[\rm (c)] each prime factor of $n$ is congruent to 1 modulo 6;
\item[\rm (d)] each solution $a$ to (\ref{eq:3rd}) gives rise to a 6-valent first-kind Frobenius circulant $TL_n(a,b,1)$, and vice versa; in this case we have $b \equiv a-1~\mod n$ and $TL_n(a,a-1,1)$ is a rotational, geometric, $\Z_n \rtimes H$-arc-transitive and first-kind $\Z_n \rtimes H$-Frobenius graph admitting $[a]$ and $-[a^2]$ as complete rotations, where 
\begin{equation}
\label{eq:H}
H = \la [a] \ra = \{\pm [1], \pm [a], \pm [a^2]\} = \{\pm [1], \pm [a], \pm [a-1]\} 
\le \Z_n^*;
\end{equation}
\item[\rm (e)] there are exactly $2^{{l-1}}$ pairwise non-isomorphic 6-valent first-kind Frobenius circulants of order $n$, where $l$ is the number of distinct prime factors of $n$, and each of them is isomorphic to $TL_n(a,a-1,1)$ for some $a$ as above.
\end{itemize}
\end{thm}
\begin{proof}
(i)~~Suppose first that there exists a first-kind Frobenius circulant $TL_n(a,b,1)$ of order $n$ such that the kernel of the underlying Frobenius group is cyclic. Then there exists a subgroup $H$ of $\Z_n^*$ such that $|H| = 6$, $\Z_n \rtimes H$ is a Frobenius group and $TL_n(a,b,1)$ is a first-kind $\Z_n \rtimes H$-Frobenius circulant. Thus $H$ is semiregular on $\Z_n \setminus \{[0]\}$ and so $n \equiv 1~\mod 6$. Moreover, $S = \{\pm [1], \pm [a], \pm [b]\}$ is an $H$-orbit on $\Z_n$ and hence $H$ is regular on $S$. Since $[1] \in S$, it follows that $S = H$. Since $H$ is Abelian with $|H| = 6$, it must be a cyclic group of order 6. So we may assume $H = \la [h] \ra = \{[1], [h], [h^2], [h^3], [h^4], [h^5]\}$ for an element $[h]$ of $\Z_n^*$ with order 6. Since $S = H$, there exists $1 \le i \le 5$ such that $[h^i] = -[1]$. Hence $[h^{2i}] = [1]$ and so $6$ divides $2i$. Therefore, $i = 3$, $[h^3] = -[1]$ and $H = \{\pm [1], \pm [h], \pm [h^2]\}$. Since $S=H$, without loss of generality we may assume $[a] = [h]$, so that $a^3 + 1 \equiv 0~\mod n$, $b \equiv a^2~\mod n$ and $H = \{\pm [1], \pm [a], \pm [a^2]\}$. Since $H$ is semiregular on $\Z_n \setminus \{[0]\}$, by Lemma \ref{semiregular} we have $[-a-1] \in \Z_n^*$ and hence $\gcd(a+1, n) = 1$. Since $a^3+1 = (a+1)(a^2-a+1) \equiv 0~\mod n$, $a$ is a solution to (\ref{eq:3rd}). Hence $b \equiv a^2 \equiv a-1~\mod n$ and $TL_n(a,b,1) = TL_n(a,a-1,1)= TL_{n}(n-a, a-1, 1)$ is geometric. It is readily seen that $[a]$ and $-[a^2]$ are complete rotations of $TL_n(a,a-1,1)$. Hence $TL_n(a,a-1,1)$ is rotational as well. Moreover, $TL_n(a, a-1, 1)$ is $\Z_n \rtimes H$-arc-transitive by \cite[Lemma 2.1]{Z}.

(ii)~~Now suppose $n \equiv 1~\mod 6$ and (\ref{eq:3rd}) is solvable. Write $n = p_1^{e_1}p_2^{e_2} \ldots p_l^{e_l}$, where $p_1, p_2, \ldots, p_l \ge 5$ are distinct primes and $e_1, e_2, \ldots, e_l \ge 1$ are integers.  
Let $a$ be a solution to (\ref{eq:3rd}). We first prove $\gcd(a+1, n) = 1$. Suppose otherwise. Then $\gcd(2a-1, n) > 1$ as $a(a+1) \equiv 2a-1~\mod n$ by (\ref{eq:3rd}). Without loss of generality we may assume that $p_1$ divides $\gcd(2a-1, n)$. Since $a$ is a solution to (\ref{eq:3rd}), by \cite[Section 2.5]{NZ} it  is also a solution to $x^2 - x +1 \equiv 0~\mod p_1^{e_1}$. Since $n$ is odd, it follows that $a$ is a solution to $4x^2 - 4x +4 \equiv 0~\mod p_1^{e_1}$, that is, $(2x-1)^2 \equiv -3~\mod p_1^{e_1}$. Moreover, $a = (p_1^{e_1} + 1)(v+1)/2~\mod p_1^{e_1}$ for some integer $v$ satisfying $v^2 \equiv -3~\mod p_1^{e_1}$. Thus $2a-1 \equiv v~\mod p_1^{e_1}$ and so $p_1$ divides $v$. Since $p_1$ divides $v^2+3$, it follows that $p_1$ divides $3$, which contradicts the fact $p_1 \ge 5$. Therefore, any solution $a$ to (\ref{eq:3rd}) satisfies $\gcd(a+1, n) = 1$. 

Since $a$ satisfies (\ref{eq:3rd}), it also satisfies $a^3 + 1 \equiv 0~\mod n$. Hence $\gcd(a, n) = 1$ and $H := \la [a] \ra = \{\pm [1], \pm [a], \pm [a^2]\} \le \Z_n^*$. Since $a^3 \equiv -1~\mod n$, $a^6 \equiv 1~\mod n$ and so the order of $[a]$ in $\Z_n^*$ is a divisor of $6$. Obviously, $a^3 \not \equiv 1~\mod n$. If $a^2 \equiv 1~\mod n$, then since $\gcd(a+1, n) = 1$ we have $a \equiv 1~\mod n$, which is a contradiction. Thus $[a]$ must have order 6 in $\Z_n^*$ and hence $|H| = 6$. Since $a$ satisfies (\ref{eq:3rd}), it follows that $\gcd(a^2+1, n) = \gcd(a, n) = 1$ and $\gcd(a-1, n) = \gcd(a^2, n) = 1$. Since $\gcd(a+1, n) = 1$, we have $\gcd(a^2-1, n) = 1$. Therefore, $H$ is semiregular on $\Z_n \setminus \{[0]\}$ by Lemma \ref{semiregular}. Hence $\Z_n \rtimes H$ is a Frobenius group. Set $S := H$. Then $S$ is an $H$-orbit on $\Z_n \setminus \{[0]\}$ and $\la S \ra = \Z_n$ since $[1] \in S$. Hence $\Cay(\Z_n, S)=TL_n(a, b, 1)$ is a 6-valent first-kind $\Z_n \rtimes H$-Frobenius graph, where $b \equiv a^2 \equiv a - 1~\mod n$ and obviously the kernel of $\Z_n \rtimes H$ is cyclic.  

(iii)~~Suppose $n \equiv 1~\mod 6$ and (\ref{eq:3rd}) is solvable. Then all statements in (d) are true by the proof above. Moreover, by \cite[Section 2.5]{NZ}, $a$ is a solution to (\ref{eq:3rd}) if and only if it is a solution to the set of congruence equations:
\begin{equation}
\label{eq:set}
x^2 - x +1 \equiv 0~~\mod p_i^{e_i},\; i = 1, 2, \ldots, l.
\end{equation}
Since $4$ is co-prime to $p_i$, we have: $a$ is a solution to (\ref{eq:3rd}) $\Leftrightarrow$ $a$ is a solution to $4x^2 - 4x +4 \equiv 0~\mod p_i^{e_i}$, i.e. $(2x-1)^2 \equiv -3~\mod p_i^{e_i}$ for $i = 1, 2, \ldots, l$ $\Leftrightarrow$ $a$ is a solution to $2x-1 \equiv v~\mod p_i^{e_i}$, for $i = 1, 2, \ldots, l$, where $v$ is a solution to $x^2 \equiv -3~\mod p_i^{e_i}$ for each $i$. Since $p_i$ is odd, for any integer $v$, $2x-1 \equiv v~\mod p_i^{e_i}$ has a unique solution, namely $(p_i^{e_i} + 1)(v+1)/2~\mod p_i^{e_i}$. Since $\gcd((p_i^{e_i} +1)/2, p_i^{e_i}) = \gcd(2, p_i^{e_i}) = 1$, the solutions of $x^2 - x +1 \equiv 0~\mod p_i^{e_i}$ and that of $x^2 \equiv -3~\mod p_i^{e_i}$ are in one-to-one correspondence. 

Since (\ref{eq:3rd}) is solvable by our assumption, from the argument above $x^2 \equiv -3~\mod p_i^{e_i}$ is solvable for each $i$. By \cite[Proposition 4.2.3]{IR} and the fact $p_i \ge 5$, this implies that $x^2 \equiv -3~\mod p_i$ is solvable, that is, $\Legendre{-3}{p_i} = 1$ for each $i$. Since $\Legendre{-3}{p_i}=(-1)^{\frac{p_i-1}{2}}\Legendre{3}{p_i}$, it follows that either $p_i \equiv 1~\mod 4$ and $\Legendre{3}{p_i} = 1$ or $p_i \equiv -1~\mod 4$ and $\Legendre{3}{p_i} = -1$. By \cite[Theorem 2, Chapter 5]{IR}, $\Legendre{3}{p_i} = 1$ if and only if $p_i \equiv \pm d^2~\mod 12$, where $d$ is an odd integer co-prime to $3$. Thus, $\Legendre{3}{p_i} = 1$ if and only if $p_i \equiv \pm 1~\mod 12$, and $\Legendre{3}{p_i} = -1$ if and only if $p_i \equiv \pm 5~\mod 12$. Therefore, $p_i \equiv 1~\mod 6$ for $i = 1, 2, \ldots, l$ and (c) holds.  

Let $N(n)$ denote the number of solutions to (\ref{eq:3rd}) and $\hat{N}(n)$ the number of solutions to $x^2 \equiv -3$~mod $n$. Then $N(n) = \prod_{i=1}^{l}N(p_i^{e_i})$ and $\hat{N}(n) = \prod_{i=1}^{l}\hat{N}(p_i^{e_i})$ by \cite[Theorem 2.18]{NZ}. Note that $N(p_i^{e_i}) = \hat{N}(p_i^{e_i})$ by our discussion above. Hence $N(n) = \hat{N}(n)$. Since $p_i \ge 5$ is not a divisor of 2 or $-3$, by \cite[Proposition 4.2.3]{IR} we have $\hat{N}(p_i^{e_i}) = N(p_i)$ for each $i$. Since $p_i \equiv 1~\mod 6$ as proved above, we have $\hat{N}(p_i) \ge 1$ as shown earlier. Thus, by \cite[Corollary 2.28]{NZ}, $\hat{N}(p_i) = 2$ and hence $N(n) = \hat{N}(n) = \prod_{i=1}^l \hat{N}(p_i^{e_i}) = 2^l$. For each solution $a$ to (\ref{eq:3rd}), we have $(-a)^2 - (-a)+1 \equiv 2a \not \equiv 0~\mod n$, $(a^2)^2 - a^2+1 \equiv -a^2 - a +1 \equiv -2(a-1) \not \equiv 0~\mod n$ and $(-a^2)^2 - (-a^2)+1 \equiv a^2 - a +1 \equiv 0~\mod n$. Hence $-a^2$ is also a solution to (\ref{eq:3rd}) and moreover no residue class in $S = \{\pm [1], \pm [a], \pm [a^2]\}$ other than $[a]$ and $-[a^2]$ is a solution to (\ref{eq:3rd}). It is proved in \cite[Theorem 4.2]{Li} that any Abelian group $G$ is a connected $p$-CI group, where $p$ is the least prime factor of $|G|$. Applying this to $\Z_n$ and noting that all $p_i \ge 7$, it follows that $\Z_n$ is a connected 7-CI group. Therefore, if $TL_n(a_1, b_1, 1) \cong TL_n(a_2, b_2, 1)$ for two solutions $a_1, a_2$ to (\ref{eq:3rd}) (where $b_1 \equiv a_1^2, b_2 \equiv a_2^2~\mod n$), then there exists $[m] \in \Z_n^*$ such that $S_1 [m] = S_2$, where $S_1 = \{\pm [1], \pm [a_1], \pm [a_1^2]\}$ and $S_2 = \{\pm [1], \pm [a_2], \pm [a_2^2]\}$. Since $[1] \in S_2$, there exists $[x] \in S_1$ such that $[xm] = [1]$. Since $S_1$ is a subgroup of $\Z_n^*$, this implies $[m] \in S_1$ and consequently $S_2 = S_1 [m] = S_1$. Note that the solutions $a$ and $-a^2$ to (\ref{eq:3rd}) give rise to the same graph. Therefore, there are exactly $2^{l-1}$ pairwise non-isomorphic 6-valent first-kind Frobenius circulants of order $n$. 
\qed
\end{proof} 

\medskip
\textsf{C. Convention and remarks.}~~Whenever we mention a 6-valent first-kind Frobenius circulant $TL_n(a,a-1,1)$ we assume without mentioning that it is as in Theorem \ref{existence} so that the kernel of the underlying Frobenius group is isomorphic to $\Z_n$.

\begin{remark} 
\label{rem:prac}
{\em (a) Note that the connection set of $TL_n(a,a-1,1)$ in Theorem \ref{existence} is the same as the complement $H$ of the underlying Frobenius group. 

(b) Since $\gcd(a, n) = \gcd(a-1, n) = \gcd(1, n) = 1$, $TL_n(a,a-1,1)$ can be decomposed into three edge-disjoint Hamilton cycles. Note that we also have $\gcd(a+1, n) = 1$ and so $\gcd(a-2, n) = 1$ as $a-2 \equiv (a-1)(a+1)~\mod n$. One can verify that $\gcd(2a-1, n) = 1$. These observations will be used in \S \ref{sec:cover}. 

(c) By Theorem \ref{existence}, a necessary condition for the existence of a 6-valent first-kind Frobenius circulant of order $n$ is $\phi(n) \equiv 0~\mod 6$, where $\phi(n)$ is Euler's totient function. 

(d) Given solutions to $x^2 \equiv -3~\mod p$ as input, there exist efficient algorithms to compute solutions to $x^2 \equiv -3~\mod p^e$ for any prime $p$ and integer $e \ge 1$. Once we work out solutions to $x^2 \equiv -3~\mod p_i^{e_i}$ for all $i$, we can find all solutions $v$ to $x^2 \equiv -3~\mod n$ by using a standard procedure (see e.g.~\cite[Section 2.5]{NZ}) based on the Chinese Remainder Theorem. 

Using the argument after (\ref{eq:set}) and noting that $n$ is odd, we see that $a$ is a solution to (\ref{eq:3rd}) if and only if $a \equiv (n+1)(v+1)/2~\mod n$ for a solution $v$ to $x^2 \equiv -3~\mod n$. So we obtain the following solution $a$ to (\ref{eq:3rd}) from $v$ (the other solution from $v$ is $-a^2 \equiv n-a+1~\mod n$):
\begin{equation}\label{eq:a}
a \equiv \left\{ 
\begin{array}{ll}
\frac{v+1}{2},\; & \mbox{if $v$ is odd} \\ [0.3cm]
\frac{n+v+1}{2},\; & \mbox{if $v$ is even.} 
\end{array}
\right.
\end{equation}
In this way we can find all solutions $a$ to (\ref{eq:3rd}) and hence all 6-valent first-kind Frobenius circulants $TL_n(a,a-1,1)$ of order $n$. Example \ref{optimal} below and Example \ref{ex:6g} in the next section show that both cases in (\ref{eq:a}) can occur. 

(e) A solution to $x^3 \equiv -1~\mod n$ may not produce a 6-valent first-kind Frobenius circulant even when every prime factor of $n$ is congruent to 1 modulo 6. For example, $12$ is a solution to $x^3 \equiv -1~\mod 91$ but not a solution to $x^2 - x + 1 \equiv 0~\mod 91$. Note that $12^2 \equiv 53~\mod 91$ but $TL_n(1, 12, 53)$ is not a first-kind Frobenius graph, for otherwise $H=\langle [12] \rangle \leq \Z_{91}^*$ would be semiregular on $\Z_{91} \setminus \{[0]\}$, which is not true as $\gcd(52,91) = 13 > 1$.}
\end{remark}

\medskip
\textsf{D. Prime power orders.}~~Theorem \ref{existence} together with its proof implies the following result.

\begin{cor}
\label{cor:unique}
Let $p$ be a prime such that $p \equiv 1~\mod 6$. Then for every integer $e \ge 1$ there is a unique 6-valent first-kind Frobenius circulant of order $p^e$, namely
$$
\Ga(p^e) =  TL_{p^e}(a_e, a_e-1, 1),
$$ 
where $a_e = (p^e + 1)(v+1)/2~\mod p^e$ with $v$ a solution to $x^2 \equiv -3~\mod p^e$.
\end{cor}

This case is both interesting and significant because, as we will see in Theorem \ref{thm:quo}, every 6-valent first-kind Frobenius circulant with order a multiple of $p^e$ is a topological cover of $\Ga(p^e)$. In particular, every graph in the sequence $\Ga(p), \Ga(p^2), \ldots, \Ga(p^e), \ldots$ is a topological cover of the graphs preceding it. Moreover, starting from $\Ga(p)$ we can construct $\Ga(p^e)$ recursively by constructing a solution $a_{s+1}$ to $x^2 - x + 1 \equiv 0~\mod p^{s+1}$ based on a solution $a_s$ to $x^2 - x + 1 \equiv 0~\mod p^s$, $s=1, 2, \ldots$. Using a standard procedure in number theory (see e.g.~\cite[Section 2.6]{NZ}), we have $a_{s+1} = a_s + p^s t$, $s = 1, 2, \ldots$, where $t$ is a solution to $(2a_s-1)t \equiv -(a_s^2 -a_s + 1)/p^s~\mod p$. For instance, in the case when $p=7$, we get $a_1 = 3, a_2 = 31, a_3 = 325, \ldots$, recursively.  

\medskip
\textsf{E. HARTS, or hexagonal meshes.}~~HARTS was proposed \cite{CSK} as a distributed real-time computing system, and its properties were studied in \cite{CSK, DRS, ABF}. We now explain that it belongs to the family of 6-valent first-kind Frobenius circulants. 
\begin{ex} 
\label{optimal}
Let $k \ge 2$ be an integer and $n_k=3k^2+3k+1$. It was proved in \cite[Theorem 1]{TZ1} that $TL_{n_k} = TL_{n_k} (3k+2, 3k+1, 1)$ ($=TL_{n_k} (1, 3k+1,-(3k+2))$) is a 6-valent first-kind Frobenius circulant. This is now an immediate consequence of Theorem \ref{existence}, because $v = 6k+3$ is a solution to $x^2 \equiv -3~\mod n$ and it gives rise to the solution $a=(v+1)/2=3k+2$ to (\ref{eq:3rd}).

It is known \cite{YFMA} that $TL_{n_k}$ has the maximum possible order among all 6-valent geometric circulants of diameter $k$. Optimal gossiping and routing schemes, and broadcasting and embedding properties of $TL_{n_k}$ have been studied in \cite{TZ1} and \cite{LOZ}, respectively.

The HARTS $H_k$ of size $k$ has diameter $k-1$ and $n_{k-1} = 3k^2 - 3k + 1$ vertices \cite{CSK}, and is isomorphic \cite{CSK, ABF} to the circulant $\Cay(\Z_{n_{k-1}}, S)$ with $S = \{\pm [k-1], \pm [k], \pm [2k-1]\}$, where the residue classes are modulo $n_{k-1}$. Since $[3k] \in \Z^*_{n_{k-1}}$, we have $\Cay(\Z_{n_{k-1}}, S) \cong \Cay(\Z_{n_{k-1}}, S')$ via the isomorphism $[x] \mapsto [3k] [x]$, where $S' = \{\pm [3k] [k-1], \pm [3k] [k], \pm [3k] [2k-1]\} = \{\pm [1], \pm [3k-1], \pm [3k-2]\}$. It follows that $H_k$ is isomorphic to $TL_{n_{k-1}}$. 

In \cite{ABF} it was noted that $H_k$ is isomorphic to the EJ graph $EJ_{k + (k-1)\rho}$. Thus $TL_{n_{k-1}}$ is isomorphic to $EJ_{k + (k-1)\rho}$. This is not a coincidence: we will see in Theorem \ref{thm:EJ} that any EJ graph $EJ_{a+b \rho}$ with $\gcd(a, b) = 1$ is isomorphic to a 6-valent first-kind Frobenius circulant. 
\qed \end{ex}

\section{Frobenius versus Eisenstein-Jacobi}
\label{sec:EJ}

In this section we prove that 6-valent first-kind Frobenius graphs form a (proper) subfamily of the family of Eisenstein-Jacobi graphs. This result enables us to obtain the distance distribution of the former from that of the latter \cite{FB}. We will also show that all Eisenstein-Jacobi graphs are arc-transitive. 

\medskip
\textsf{A. Eisenstein-Jacobi graphs.}~~Let $\rho = (1+\sqrt{-3})/2$ and let $\Z[\rho] = \{x + y\rho: x, y \in \Z\}$ be the ring of Eisenstein-Jacobi integers \cite{IR}. It is well-known \cite{IR} that $\Z[\rho]$ is a Euclidean domain with norm defined by $N(x + y\rho) = x^2+xy+y^2$. We have $\rho^2 - \rho + 1 = 0, \rho^3 = -1$ and the set of units of $\Z[\rho]$ is $\{\rho^j: j \in \Z\} = \{\pm 1, \pm \rho, \pm \rho^2\} = \{\pm 1, \pm \rho, \pm (\rho-1)\}$.  

Let $0 \ne \a = c+d\rho \in \Z[\rho]$. 
Consider the quotient ring $\Z[\rho]/(\a)$ of $\Z[\rho]$ with respect to the principal ideal $(\a)$. For any $\eta \in \Z[\rho]$, let $[\eta]_{\a} \in \Z[\rho]/(\a)$ denote the residue class containing $\eta$ modulo $\a$. If $N(\a) \ge 7$, the {\em Eisenstein-Jacobi graph} (or {\em EJ graph} for short) $EJ_{\a}$ generated by $\a$ is defined \cite{MBG} as the Cayley graph on the additive group of $\Z[\rho]/(\a)$ with respect to $\{\pm [1]_{\a}, \pm [\rho]_{\a}, \pm [\rho^2]_{\a}\}$. The assumption $N(\a) \ge 7$ ensures that $\pm [1]_{\a}, \pm [\rho]_{\a}, \pm [\rho^2]_{\a}$ are pairwise distinct and so $EJ_{\a}$ is a 6-valent graph with $N(\a)$ vertices.  

Instead of $\Z[\rho]$, in \cite{MBG} EJ graphs are defined on $\Z[\omega]$ with norm $N(x + y\omega) = x^2-xy+y^2$, where $\omega = (-1+\sqrt{-3})/2$. Although $EJ_{c+d\om}$ defined on $\Z[\omega]$ in this way has $c^2 - cd + d^2$ vertices and is different from our graph $EJ_{c+d\rho}$, the family of EJ graphs is the same \cite{FB} no matter whether $\Z[\rho]$ or $\Z[\omega]$ is used, and all results for EJ graphs on $\Z[\omega]$ can be translated into results for EJ graphs on $\Z[\rho]$. Our terminology in this and the next sections agrees with that in \cite{FB}.  

\begin{lem}
\label{lem:iso} (\cite[Theorem 20]{MBG}) 
Let $\a = c+d\rho \in \Z[\rho]$ be such that $N(\a) \ge 7$ and $\gcd(c, d) = 1$. Denote $n = N(\a)$. Then 
$$
EJ_{\a} \cong TL_{n}(c, d, c+d).
$$
\end{lem}

In fact, since $\gcd(c,d)=1$, any integer can be expressed as $dx-cy$ for some integers $x$ and $y$. One can verify that 
\begin{equation}
\label{eq:isom}
\Z_{N(\a)} \rightarrow \Z[\rho]/(\a),\, dx-cy~\mod n \mapsto [x+y\rho]_{\a},\, x, y \in \Z
\end{equation}
defines the required isomorphism from $TL_{n}(c, d, c+d)$ to $EJ_{\a}$. 

\medskip
\textsf{B. 6-valent first-kind Frobenius circulants are EJ graphs.} 

\begin{thm}
\label{thm:EJ} 
\begin{itemize}
\item[\rm (a)] 
Every 6-valent first-kind Frobenius circulant $TL_{n}(a, a-1, 1)$ is isomorphic to some $EJ_{\a}$ with $\a = c+d\rho$ satisfying $\gcd(c,d)=1$. 
 
Moreover, letting $k$ be the integer defined by $a^2 - a +1 = kn$, we have $\a = (rn+m)+(sn-ma)\rho$, $(rn+m)+(sn+m(a-1))\rho$ or $(rn+ma)+(sn-m(a-1))\rho$, where $(m, r, s)$ is a solution to one of the following Diophantine equations, respectively,
\begin{equation}
\label{eq:cs1}
km^2-[(a-2)r+(2a-1)s]m+(r^2+rs+s^2)n=1
\end{equation} 
\begin{equation}
\label{eq:cs2}
km^2+[(a+1)r+(2a-1)s]m+(r^2+rs+s^2)n=1
\end{equation}
\begin{equation}
\label{eq:cs3}
km^2+[(a+1)r-(a-2)s]m+(r^2+rs+s^2)n=1.
\end{equation}

\item[\rm (b)] 
Let $\a = c+d\rho \in \Z[\rho]$ be such that $N(\a) \ge 7$ and $\gcd(c, d) = 1$. Then $EJ_{\a}$ is isomorphic to a 6-valent first-kind Frobenius circulant if and only if $N(\a) \equiv 1~\mod 6$. 
\end{itemize}
\end{thm}

\begin{proof} 
(a) Let $TL_{n}(a, a-1, 1)$ be a 6-valent first-kind Frobenius circulant, where $n \ge 7$, $n \equiv 1~\mod 6$, and $a$ is a solution to (\ref{eq:3rd}). Define $f: \Z[\rho] \rightarrow \Z_n$ by $f(x+y\rho) = [x+ya]$. Since $a^2 - a + 1 \equiv 0~\mod n$, one can verify that $f$ is a well-defined ring homomorphism. Since $\Z[\rho]$ is a Euclidean domain, it is a principal ideal domain. Thus the kernel of $f$ must be a principal ideal of $\Z[\rho]$; that is, $\ker(f) = (\a)$ for some $0 \ne \a = c+d\rho \in \Z[\rho]$. Since $\Z[\rho]/(\a) \cong \Z_n$ and $f$ maps $\{\pm 1, \pm \rho, \pm (\rho-1)\}$ to $\{\pm[1], \pm[a], \pm[a-1]\}$, we obtain $EJ_{\a} \cong TL_{n}(a, a-1, 1)$.  

Thus $N(\a) = n$, and by Lemma \ref{lem:iso}, $TL_{n}(a, a-1, 1) \cong TL_{n}(c, d, c+d)$. Since $\Z_n$ is a $7$-CI-group by \cite[Theorem 4.2]{Li}, it follows that there exists an integer $m$ with $\gcd(m,n)=1$ such that $\{[ma], [m(a-1)], [m], -[ma], -[m(a-1)], -[m]\} = \{[c], [d], [c+d], -[c], -[d], -[c+d]\}$. 
Since $2, a-1$ and $a$ are all coprime to $n$, we have $\{[c], [d]\} = \{[m], -[ma]\}$, $\{-[m], [ma]\}$, $\{[m], [m(a-1)]\}$, $\{-[m], -[m(a-1)]\}$, $\{[ma], -[m(a-1)]\}$ or $\{-[ma], [m(a-1)]\}$. Since the roles of $[c]$ and $[d]$ are symmetric and $TL_{n}(c, d, c+d) = TL_{n}(-c, -d, -c-d)$, it suffices to consider three cases: $([c], [d]) = ([m], -[ma])$, $([m], [m(a-1)])$ or $([ma], -[m(a-1)])$. 

In the case when $([c], [d]) = ([m], -[ma])$, there exist integers $r$ and $s$ such that $c=rn+m, d=sn-ma$, $\gcd(rn+m, sn-ma)=1$ and $n=(rn+m)^2+(rn+m)(sn-ma)+(sn-ma)^2$. This is equivalent to saying that $(m, r, s)$ is a solution to (\ref{eq:cs1}). One can verify that $(m, r, s)$ satisfies $\gcd(m,n)=1$ and $\gcd(c, d)=1$. The other two cases can be treated similarly. 

(b) Denote $n = N(\a) = c^2 + cd + d^2$. The necessity follows from Theorem \ref{existence} immediately. 

To prove the sufficiency, suppose $n \equiv 1~\mod 6$. Since $\gcd(c, d) = 1$ and $n = c^2 + cd + d^2$, we have $[d] \in \Z_n^*$. Let $[g]$ be the inverse of $[d]$ in $\Z_n^*$ and let $a \equiv -cg~\mod n$ be such that $0 \le a \le n-1$. Multiplying $c^2 + cd + d^2 = n$ by $g^2$, we obtain $a^2 - a + 1 \equiv 0~\mod n$. Thus, by Theorem \ref{existence}, $TL_{n}(a, a-1, 1)$ is a 6-valent first-kind Frobenius circulant. By Lemma \ref{lem:iso} and noting $\gcd(g, n) = 1$, we obtain $EJ_{\a} \cong TL_{n}(c, d, c+d) \cong TL_{n}(-cg, -dg, -(c+d)g) \cong TL_{n}(a, a-1, 1)$. 
\qed 
\end{proof}

As we will see in Examples \ref{third} and \ref{ex:6g}, a 6-valent first-kind Frobenius circulant may be isomorphic to two EJ graphs $EJ_{\a}, EJ_{\a'}$ with $\a \ne \a'$, and $\a, \a'$ can be solutions of different equations among (\ref{eq:cs1})--(\ref{eq:cs3}).   

\medskip
\textsf{C. Distance distribution and examples.}~~Given a Cayley graph $\Ga$ and integer $t \ge 0$, let $W_t(\Ga)$ denote the number of vertices in $\Ga$ whose distance to the identity element (or any other fixed element) of the underlying group is equal to $t$. In \S \ref{sec:rg} we will need the values of these parameters for a 6-valent first-kind Frobenius graph. Theorem \ref{thm:EJ} enables us to obtain such information by using the following known result.  

\begin{thm}
\label{thm:fb}
(\cite[Theorem 27]{FB}) Let $\a = c+d\rho \in \Z[\rho]$ be such that $\a \ne 0$ and $c \ge d \ge 0$. Then
$$
W_t(EJ_{\a}) = \left\{ 
\begin{array}{ll}
1, & t = 0\\[0.2cm]
6t, & 1 \le t < (c+d)/2\\[0.2cm]
6(2c+d) - 18t, & (c+d)/2 < t < (2c+d)/3\\[0.2cm]
2, & c \equiv d~\mod 3\;\,\mbox{and}\;\,t = (2c+d)/3\\[0.2cm]
0, & t > (2c+d)/3.
\end{array} 
\right.
$$
In particular, the diameter of $EJ_{\a}$ is equal to $\lfloor (2c+d)/3 \rfloor$.
In addition, if $c+d=2t^*$ is even, then $W_{t^*}(EJ_{\a})$ is equal to $c^2 + cd + d^2$ minus the total number of vertices listed above.  
\end{thm}

Theorem \ref{thm:fb} covers all EJ graphs since any EJ graph is \cite{FB} isomorphic to some $EJ_{c+d\rho}$ with $c \ge d \ge 0$. This is because \cite{FB} $EJ_{\rho^j \a} \cong EJ_{\a}$ for every integer $j$ and $EJ_{c+d\rho} \cong EJ_{d+c\rho}$.  

We illustrate Theorems \ref{thm:EJ} and \ref{thm:fb} by the following examples. 

\begin{ex}
\label{third}
Let $a \ge 3$ be an integer such that all prime factors of $n = a^2 - a +1$ are congruent to 1 modulo 6. (Hence $a \not \equiv 2~\mod 3$.) Then $\Ga = TL_{n}(a, a-1, 1)$ is a 6-valent first-kind Frobenius circulant. Here $k = 1$ as $n = a^2 - a +1$, where $k$ is as in Theorem \ref{thm:EJ}. It can be verified that $(m, r, s) = (1, 0, 0)$ is a solution to each of (\ref{eq:cs1})--(\ref{eq:cs3}). Thus, by Theorem \ref{thm:EJ}, $\Ga \cong EJ_{1-a\rho} \cong EJ_{1+(a-1)\rho} \cong EJ_{a-(a-1)\rho}$.  

Regarding $\Ga \cong EJ_{(a-1)+\rho}$ as an EJ graph allows us to compute its distance distribution. Since $a \not \equiv 2~\mod 3$, by Theorem \ref{thm:fb} the diameter of $\Ga$ is $D=\lfloor (2a-1)/3 \rfloor$. Moreover, $W_t(\Ga) = 6t$ for $1 \le t < a/2$ and $W_t(\Ga) = 6((2a-1)-3t)$ for $a/2 < t \le D$. In addition, if $a$ is even, then $W_{a/2}(\Ga) = n - 1 - \sum_{t \ne a/2} W_t(\Ga) = n-1-\frac{3a(a-2)}{4}+\left(3D-\frac{3a}{2}\right)\left(3D-\frac{5a}{2}+5\right)$.  
\qed \end{ex} 

\begin{ex}
\label{ex:6g}
Let $n = 12g^2 + 1$ where $g \ge 1$ is an integer. Then $v = 6g$ is a solution to 
$x^2 \equiv -3~\mod n$ and it gives rise to the solution $a = (n+v+1)/2 = 6g^2 + 3g + 1$ to (\ref{eq:3rd}) (see Remark \ref{rem:prac}(d)). Thus $\Ga = TL_{n}(6g^2 + 3g + 1, 6g^2 + 3g, 1)$ is a 6-valent first-kind Frobenius circulant. 

Note that $a^2-a+1 = (3g^2 + 3g + 1)n$. Hence $k = 3g^2 + 3g + 1$. Since $(m, r, s) = (2g-1, 0, g)$ is a solution to (\ref{eq:cs1}), by Theorem \ref{thm:EJ}, $\Ga \cong EJ_{\a}$, where $\a = (2g-1)+[gn-(2g-1)a]\rho = (2g-1)+(2g+1)\rho$. From (\ref{eq:isom}) and the proof of Theorem \ref{thm:EJ}, $\Z_{n} \rightarrow \Z[\rho]/(\a), [u] \mapsto [-u\rho]_{\a}$ defines an isomorohism from $\Ga$ to $EJ_{\a}$. 

It can be verified that $(m, r, s) = (2g-1, 0, -g)$ is a solution to (\ref{eq:cs2}). From this we get $\Ga \cong EJ_{\b}$, where $\b = (2g-1)+[-gn+(2g-1)(a-1)]\rho = (2g-1)-4g\rho$, and $\Z_{n} \rightarrow \Z[\rho]/(\b), [u] \mapsto [-u\rho]_{\b}$ gives the required isomorphism.
\qed \end{ex}

\textsf{D. EJ graphs are arc-transitive.}~~Since any Cayley graph is vertex-transitive, all EJ graphs are vertex-transitive. We now prove that they are actually arc-transitive. (An arc-transitive graph without isolated vertices is vertex-transitive, but the converse is not true.) The proof is similar to that of a counterpart result \cite[Lemma 7]{Z1} for Gaussian graphs \cite{MBG}. 

\begin{thm}
\label{thm:arc tran}
Let $\a \in \Z[\rho]$ be such that $N(\a) \ge 7$. Let
$$
H_{\a} = \{\pm [1]_{\a}, \pm [\rho]_{\a}, \pm [\rho^2]_{\a}\}.
$$   
Then $(\Z[\rho]/(\a)) \rtimes H_{\a}$ is isomorphic to a subgroup of the automorphism group of $EJ_{\a}$, and $EJ_{\a}$ is $(\Z[\rho]/(\a)) \rtimes H_{\a}$-arc-transitive. 
\end{thm}

\begin{proof}  
$H_{\a}$ is a group under the multiplication of $\Z[\rho]/(\a)$. It can be verified that 
$(x+y\rho)^{\rho^j} = (x+y\rho)\rho^j$ 
defines an action of $H_{\a}$ on the additive group of $\Z[\rho]/(\a)$ (as a group \cite{Dixon-Mortimer}). Here and in the rest of this proof an EJ integer is interpreted as its residue class modulo $\a$. Thus the semidirect product $(\Z[\rho]/(\a)) \rtimes H_{\a}$ is well-defined. Moreover, it acts on $\Z[\rho]/(\a)$ (as a set) by
$(x+y\rho)^{(c+d\rho, \rho^j)} = ((x+c)+(y+d)\rho)\rho^j$
for $x+y\rho \in \Z[\rho]/(\a)$ and $(c+d\rho, \rho^j) \in (\Z[\rho]/(\a)) \rtimes H_{\a}$. It can be verified that this action is faithful, that is, the only element of $(\Z[\rho]/(\a)) \rtimes H_{\a}$ that fixes every $x+y\rho \in \Z[\rho]/(\a)$ is its identity element $(0, 1)$.
It can also be verified that $(\Z[\rho]/(\a)) \rtimes H_{\a}$ preserves adjacency and non-adjacency relations of $EJ_{\a}$. Hence $(\Z[\rho]/(\a)) \rtimes H_{\a}$ is isomorphic to a subgroup of the automorphism group of $EJ_{\a}$. 

Let $x_t + y_t \rho$ and $u_t + v_t \rho$ be adjacent in $EJ_{\a}$, $t = 1, 2$. Then $x_t + y_t \rho =  (u_t + v_t \rho) + \rho^{i_t}$ for some integer $i_t$. It is straightforward to verify that the element $((u_2 + v_2 \rho)\rho^{i_1 - i_2} - (u_1 + v_1 \rho), \rho^{i_2 - i_1})$ of $(\Z[\rho]/(\a)) \rtimes H_{\a}$ maps arc $(x_1 + y_1 \rho, u_1 + v_1 \rho)$ to arc $(x_2 + y_2 \rho, u_2 + v_2 \rho)$. Since this holds for any two arcs, $EJ_{\a}$ is $(\Z[\rho]/(\a)) \rtimes H_{\a}$-arc-transitive.
\qed
\end{proof}

\section{Covers and recursive constructions}
\label{sec:cover}

Let $\Ga_1$ and $\Ga_2$ be graphs. We say that $\Ga_1$ is a {\em cover} of $\Ga_2$ if there exists a surjective mapping $\phi: V(\Ga_1) \rightarrow V(\Ga_2)$ such that for each $u \in V(\Ga_1)$, the restriction of $\phi$ to the neighbourhood $N_{1}(u)$ of $u$ in $\Ga_1$ is a bijection from $N_{1}(u)$ to the neighbourhood $N_{2}(\phi(u))$ of $\phi(u)$ in $\Ga_2$. If in addition $k = |\phi^{-1}(v)|$ for all $v \in V(\Ga_2)$, then we say that $\Ga_1$ is a {\em $k$-fold cover} of $\Ga_2$.

Let $\Ga$ be a graph and $\cal P$ a partition of $V(\Ga)$. The {\em quotient graph} of $\Ga$ with respect to $\cal P$, $\Ga_{\cal P}$, is defined to have vertex set $\cal P$ such that $P_1, P_2 \in {\cal P}$ are adjacent if and only if there exists an edge of $\Ga$ joining a vertex of $P_1$ to a vertex of $P_2$. Let $G$ be a group of automorphisms of $\Ga$. If for any block $P \in {\cal P}$ and any $g \in G$ the image of $P$ under $g$ is also a block of $\cal P$, then $\cal P$ is called {\em $G$-invariant}. It can be verified that, if $\Ga$ is $G$-arc-transitive and $\cal P$ is $G$-invariant, then $\Ga_{\cal P}$ is also $G$-arc-transitive. 

\medskip
\textsf{A. Covering EJ graphs by 6-valent first-kind Frobenius circulants.}~~The following result together with its proof is similar to that of \cite[Lemma 8]{Z1}. 

\begin{thm}
\label{thm:cover}
Let $\a, \b \in \Z[\rho]$ be nonzero such that $N(\a) \ge 7$. Then $EJ_{\a\b}$ is an $N(\b)$-fold cover of $EJ_{\a}$ and can be constructed from $EJ_{\a}$. 
\end{thm}

\begin{proof}
Let $K = ([\a]_{\a\b})$ be the principal ideal of $\Z[\rho]_{\a\b}$ induced by $[\a]_{\a\b}$. Since $\Z[\rho]$ is an Euclidean domain, its elements are of the form $\xi = \eta \b + \d$ with $\d=0$ or $N(\d) < N(\b)$. Hence $K=\{[\a\d]_{\a\b}: \d \in \Z[\rho],\, \d=0\;\mbox{or}\;N(\d) < N(\b)\}$. Since $K = (\a)/(\a\b)$, when it is viewed as a subgroup of the additive group of $\Z[\rho]_{\a\b}$, we have $\Z[\rho]_{\a} \cong \Z[\rho]_{\a\b}/K$ via the classical isomorphism $[\xi]_{\a} \mapsto K+[\xi]_{\a\b},\, [\xi]_{\a} \in \Z[\rho]_{\a}$. Hence $|K| = N(\a\b)/N(\a) = N(\b)$.

Now we construct a graph $\hat{EJ}_{\a\b}$ with vertex set $\Z[\rho]_{\a\b}$. Consider an arbitrary pair of adjacent vertices $[\xi]_{\a}, [\xi']_{\a}$ of $EJ_{\a}$. By the definition of $EJ_{\a}$, there exist $\eta \in \Z[\rho]$ and a unit $\ve$ of $\Z[\rho]$, both relying on $\xi$ and $\xi'$, such that $\xi - \xi' = \a \eta + \ve$. Construct $\hat{EJ}_{\a\b}$ in such a way that 
\bea
\mbox{each $[\a\d+\xi]_{\a\b} \in K+[\xi]_{\a\b}$ is adjacent to $[\a\d+\xi-\ve]_{\a\b} = [\a(\d+\eta)+\xi']_{\a\b} \in K+[\xi']_{\a\b}$} \non \\ [0.15cm]
\mbox{but not any other element in $K+[\xi']_{\a\b}$.} \label{eq:const}
\eea
This adjacency relation is defined for all pairs of adjacent vertices $[\xi]_{\a}, [\xi']_{\a}$ of $EJ_{\a}$. Since $\xi' - \xi = -\a \eta - \ve$, when interchanging the roles of $[\xi]_{\a}$ and $[\xi']_{\a}$ in (\ref{eq:const}), we obtain that $[\a\d+\xi-\ve]_{\a\b} = [\a(\d+\eta)+\xi']_{\a\b}$ is adjacent to $[\a(\d+\eta)+\xi'+\ve]_{\a\b}=[\a\d+\xi]_{\a\b}$ in $\hat{EJ}_{\a\b}$. Hence the adjacency relation (\ref{eq:const}) is symmetric. Moreover, it is independent of the choice of representatives of $[\xi]_{\a}$ and $[\a\d+\xi]_{\a\b}$. In fact, if $[\a\d_1+\xi_1]_{\a\b} = [\a\d+\xi]_{\a\b}$ (which implies $[\xi_1]_{\a} = [\xi]_{\a}$), then $\xi_1 = \xi + \a(\s \b+\d-\d_1)$ for some $\s \in \Z[\rho]$ and hence $\xi_1 -\xi' = \a(\s \b+\d-\d_1+\eta)+\ve$. Thus, by (\ref{eq:const}), $[\a\d_1+\xi_1]_{\a\b} \in K+[\xi]_{\a\b}$ is adjacent to $[\a(\d_1+(\s \b+\d-\d_1 + \eta))+\xi']_{\a\b}=[\a(\d+\eta)+\xi']_{\a\b} \in K+[\xi']_{\a\b}$, which agrees with (\ref{eq:const}) applied to $[\a\d+\xi]_{\a\b}$. Therefore, $\hat{EJ}_{\a\b}$ is well-defined as an undirected graph. Since $EJ_{\a}$ is 6-valent, by the above construction, $\hat{EJ}_{\a\b}$ is 6-valent as well.  

Using the notation above, by the definition of $EJ_{\a\b}$, $[\a\d+\xi]_{\a\b}$ and $[\a\d+\xi-\ve]_{\a\b}$ are clearly adjacent in $EJ_{\a\b}$. Thus, by (\ref{eq:const}), if two vertices are adjacent in $\hat{EJ}_{\a\b}$, then they are adjacent in $EJ_{\a\b}$. This implies that $\hat{EJ}_{\a\b}$ is a spanning subgraph of $EJ_{\a\b}$. Since both graphs are 6-valent, it follows that they must be identical. 
Therefore, $EJ_{\a\b}$ can be constructed from $EJ_{\a}$ as in the previous paragraph. It is obvious that the quotient graph of $EJ_{\a\b}$ with respect to the partition $\Z[\rho]_{\a\b}/K$ of $\Z[\rho]_{\a\b}$ is isomorphic to $EJ_{\a}$, and moreover $EJ_{\a\b}$ is an $N(\b)$-fold cover of $EJ_{\a}$.
\qed
\end{proof}  

Two elements $\a, \b \in \Z[\rho]$ are said to be {\em associates} if $\a = \b \rho^j$ for some integer $j$. 

\begin{cor}
\label{cor:cover}
Let $\a = c+d\rho \in \Z[\rho]$ with $7 \le N(\a) \equiv 1~\mod 6$ that is not an associate of any real integer. Denote $\ell = \gcd(c, d)$, $c' = c/\ell$, $d' = d/\ell$ and $\a' = c' + d' \rho$. Then $EJ_{\a}$ is an $\ell^2$-fold cover of a 6-valent first-kind Frobenius circulant that is isomorphic to $EJ_{\a'}$.
\end{cor}

\begin{proof}
We have $N(\a) = \ell^2 N(\a')$ and $\ell^2 \equiv 1, 3$ or $4~\mod 6$ as $\ell \ne 0$.   
If $\ell^2 \equiv 3~\mod 6$, then $N(\a) \equiv 0$ or $3~\mod 6$, a contradiction. If $\ell^2 \equiv 4~\mod 6$, then $N(\a) \equiv 0, 2$ or $4~\mod 6$, a contradiction again. So we must have $\ell^2 \equiv 1~\mod 6$ and consequently $N(\a') \equiv 1~\mod 6$. We have $N(\a') \ge 7$, for otherwise $\a$ is an associate of the integer $\ell$, a contradiction. Similarly, we have $c \ne 0$ and $d \ne 0$. Since $\gcd(c', d') = 1$, by Theorem \ref{thm:EJ} it follows that $EJ_{\a'}$ is isomorphic to a 6-valent Frobenius circulant. Since $\a = \ell \a'$ and $N(\ell) = \ell^2$, by Theorem \ref{thm:cover} $EJ_{\a}$ is an $\ell^2$-fold cover of $EJ_{\a'}$. 
\qed
\end{proof}

\medskip
\textsf{B. Covering 6-valent first-kind Frobenius circulants.}~~Let $n \ge 7$ be an integer and $m > 1$ a divisor of $n$. Let 
$$
K(m) = \{[km]: 0 \le k \le n/m - 1\}
$$
be the subgroup of the additive group $(\Z_{n}, +)$ generated by $[m]$. Let
$$
{\cal P}(m) = \Z_{n}/K(m) = \{K(m) + [j]: 0 \le j \le m - 1\} \cong \Z_m
$$ 
be the quotient group of $(\Z_n, +)$ by $K(m)$. We may also view ${\cal P}(m)$ as a partition of $\Z_n$. 

The following result states that any 6-valent first-kind Frobenius circulant is a cover of and can be constructed from its proper `quotient' 6-valent first-kind Frobenius circulants. 

\begin{thm}
\label{thm:quo}
Let $n \ge 7$ be an integer all of whose prime factors are congruent to $1$ modulo $6$. Let $a$ be a solution to (\ref{eq:3rd}) and $H$ be as in (\ref{eq:H}), so that $\Ga = TL_{n}(a, a-1,1)$ is a 6-valent first-kind Frobenius circulant of order $n$. Then for every proper divisor $m$ of $n$, the quotient graph of $\Ga$ with respect to the partition ${\cal P}(m)$ is isomorphic to a 6-valent first-kind Frobenius circulant of order $m$, namely $\Ga(m) = TL_{m}(a_m, a_m - 1, 1)$, where $a_m$ is a solution to $x^2 - x +1 \equiv 0~\mod m$. Moreover, $\Ga$ is an $n/m$-fold cover of $\Ga(m)$.
\end{thm}

\begin{proof}
Denote $H/K(m) = \{K(m)+[1], K(m)-[1], K(m)+[a], K(m)-[a], K(m)+[a-1], K(m)-[a-1]\}$. Then $-H/K(m) = H/K(m)$ and $K(m) \not \in H/K(m)$ as $1, a, a-1$ are all coprime to $n$ (Remark \ref{rem:prac}(b)). Hence $\Cay({\cal P}(m), H/K(m))$ is a well-defined Cayley graph. It is readily seen that $\Ga_{{\cal P}(m)} \cong \Cay({\cal P}(m), H/K(m))$. Since ${\cal P}(m)$ is induced by the normal subgroup $K(m)$ of $\Z_n$, $\Ga$ must be a multicover of $\Ga_{{\cal P}(m)}$; that is, for $K(m) + [j_1], K(m) + [j_2] \in {\cal P}(m)$ adjacent in $\Ga_{{\cal P}(m)}$, every $[km+j_1] \in K(m) + [j_1]$ has the same number of neighbours in $K(m) + [j_2]$. Suppose $[j_1]$ is adjacent to distinct $[km+j_2], [k'm+j_2] \in K(m) + [j_2]$. Then $[km+j_2-j_1], [k'm+j_2-j_1] \in H$ and so $[(k-k')m]$ is equal to one of $\pm [1], \pm [2], \pm [a], \pm [2a], \pm [a-1], \pm [a+1], \pm [a-2], \pm [2a-1]$ and $\pm [2(a-1)]$. However, this is impossible because by Remark \ref{rem:prac}(b) all these numbers are coprime to $n$ and hence to $m$. This contradiction shows that every vertex in $K(m) + [j_1]$ has exactly one neighbour in $K(m) + [j_2]$. Therefore, $\Ga$ is an $n/m$-fold cover of $\Ga(m)$. 

Let $a \equiv a_m~(\mod m)$, where $1 \le a_m \le m-1$. (Note that $a_m \ne 0$ as $a$ and $m$ are coprime.) Since $a$ satisfies (\ref{eq:3rd}) and $m$ divides $n$, $a_m$ is a solution to $x^2 - x + 1 \equiv 0~(\mod m)$. Let $H(m) = \la [a]_{m} \ra \le \Z_{m}^*$, where $[x]_{m}$ denotes the residue class of $x$ modulo $m$. By Theorem \ref{existence}, $\Ga(m) = TL_{m}(a_m, a_m - 1, 1)$ is a 6-valent first-kind Frobenius circulant produced by $H(m)$. It is straightforward to verify that 
${\cal P}(m) \rightarrow \Z_{m},\;\,K(m) + [j] \mapsto [j]_{m},\;\,0 \le j \le m - 1$
is an isomorphism from $\Ga_{{\cal P}(m)}$ to $\Ga(m)$. 
\qed
\delete
{
One can see that 
$$
\phi: {\cal P}(m) \rightarrow \Z_{m},\;\,K(m) + [j] \mapsto [j]_{m},\;\,0 \le j \le m - 1
$$
is a bijection from ${\cal P}(m)$ to $\Z_{m}$. We now prove that $\phi$ is an isomorphism from $\Ga_{{\cal P}(m)}$ to $\Ga(m)$. If $K(m)+[j_1]$ and $K(m)+[j_2]$ are adjacent in $\Ga_{{\cal P}(m)}$, then $[j_1 - j_2] = [km] +[x]$ for some $k$ and $[x] \in H$. Since $a \equiv a_m~(\mod m)$, by the definition of $H(m)$, we have $x \equiv x(m)~(\mod m)$ for some $[x(m)]_{m} \in H(m)$. Hence $[j_1 - j_2]_{m} = [x(m)]_{m} \in H(m)$, and so $[j_1]_{m}$ and $[j_2]_{m}$ are adjacent in $\Ga(m)$. Conversely, if $[j_1]_{m}$ and $[j_2]_{m}$ are adjacent in $\Ga(m)$, then $[j_1 - j_2]_{m} \in H(m)$ and so $j_1 - j_2 \equiv x(m)~(\mod m)$ for some $[x(m)]_{m} \in H(m)$. By the definition of $H(m)$, $x \equiv x(m)~(\mod m)$ for some $[x] \in H$. Hence $[j_1 - j_2] \in K(m) + [x]$ and so $[j_1]$ and $[j_2]$ are adjacent in $\Ga_{{\cal P}(m)}$. Therefore, $\Ga_{{\cal P}(m)} \cong \Ga(m)$ via $\phi$.
}
\end{proof}

\section{Gossiping, routing and Wiener index}
\label{sec:rg}

In this and the next sections we study gossiping (all-to-all communication), routing and broadcasting (one-to-all communication) problems for 6-valent first-kind Frobenius circulants. We will present our results in terms of such graphs, but in view of Theorem \ref{thm:EJ} the same results can also be stated in terms of EJ graphs $EJ_{c+d\rho}$ with $\gcd(c,d) = 1$ and $N(c+d\rho) \equiv 1~\mod 6$. At present we do not know whether the same results hold for arbitrary EJ graphs since our proofs rely on properties of Frobenius groups.  

\textsf{A. Routing and gossiping.}~~
A {\em routing} of a connected graph $\Ga=(V,E)$ is a set of oriented paths, one for each ordered pair of vertices. The {\em load of an edge} with respect to a routing is the number of times it is traversed by such paths in either direction; the {\em load of a routing} is the maximum load on an edge; and the {\em edge-forward index} $\pi(\Ga)$ is \cite{HMS} the minimum load over all possible routings of $\Ga$. The {\em arc-forwarding index} $\overrightarrow{\pi}(\Ga)$ is defined \cite{Hey} similarly by taking the direction into account when counting the number of times an arc is traversed. (Recall that an arc is an ordered pair of adjacent vertices.) A routing is a {\em shortest path routing} if all paths used are shortest paths. The {\em minimal edge-} and {\em arc-forwarding indices} \cite{Hey},  $\pi_m(\Ga)$, $\overrightarrow{\pi}_m(\Ga)$, are defined by restricting to shortest path routings in the definitions of $\pi$ and $\overrightarrow{\pi}$, respectively. It is easy to see (e.g. \cite[Theorem 3.2]{HMS}) that 
\begin{equation}
\label{eq:4ineq}
\pi_m \ge \pi \ge \frac{\sum_{(u,v)\in V \times V} d(u,v)}{|E|},\quad \overrightarrow{\pi}_m \ge \overrightarrow{\pi} \ge \frac{\sum_{(u,v)\in V \times V} d(u,v)}{2|E|},
\end{equation}
where $d(u, v)$ is the distance between $u$ and $v$ in the graph. 

An information dissemination process such that each vertex has a distinct message to be sent to all other vertices is called {\em gossiping} (all-to-all communication). We consider the {\em store-and-forward, all-port and full-duplex} model \cite{BKP}: a vertex must receive a message wholly before retransmitting it to other vertices; a vertex can exchange messages (which may be different) with all of its neighbours at each time step; messages can traverse an edge in both directions simultaneously; no two messages can transmit over the same arc at the same time; and it takes one time step to transmit any message over an arc. A \emph{gossiping scheme} is a procedure fulfilling the gossiping under these constraints, and the {\em minimum gossip time} \cite{BKP} of a graph $\Ga$, denoted by $t(\Ga)$, is the minimum number of time steps required by such a scheme. Clearly, if $\Ga$ has minimum valency $\d$, then \cite{BKP}
\begin{equation}
\label{eq:t}
t(\Ga) \ge \frac{|V|-1}{\d}.
\end{equation}

\medskip
\textsf{B. Computing forwarding indices and minimum gossip time.}~~Given a first-kind $K \rtimes H$-Frobenius graph with diameter $D$, the set of vertices at distance $t$ from the identity element of $K$ is a union of $H$-orbits on $K$, $1 \le t \le D$. Denote by $n_t$ the number of such $H$-orbits, and call $(n_1, \ldots, n_D)$ the {\em type} \cite{FLP} of the graph.

In the remainder of this section, we use
\begin{equation}
\label{eq:Gamma}
\Ga = TL_{n}(a,a-1,1)
\end{equation}
to denote a 6-valent first-kind Frobenius circulant, where each prime factor of $n \ge 7$ is congruent to 1 modulo 6 and $a$ is a solution to (\ref{eq:3rd}). Let $D=\diam(\Ga)$ be the diameter of $\Ga$ and $\Ga_{t}[0]$ the set of vertices of $\Ga$ distant $t$ apart from $[0]$, $1 \le t \le D$. Then $\Ga_{t}[0]$ has size $W_{t}(\Ga)$.
Theorems \ref{thm:EJ} and \ref{thm:fb} together enable us to compute the type $(n_1, \ldots, n_D)$ of $\Ga$ in the following way. First, we work out $\a = c+d\rho$ such that $\Ga \cong EJ_{c+d\rho}$ by using Theorem \ref{thm:EJ}. Multiplying $\a$ by an appropriate $\rho^j$ and/or interchanging $c$ and $d$ when necessary, we may assume $c \ge d \ge 0$. (See the paragraph right after Theorem \ref{thm:fb}.) Since each $\Ga_{t}[0]$ is the union of $n_t$ $H$-orbits, where $H = \la [a] \ra$ as in (\ref{eq:H}), we have $W_t(\Ga) = 6n_t$ and in particular $W_t(\Ga) \ne 2$ (hence $c \not \equiv d~\mod 3$). Thus, by Theorem \ref{thm:fb}, $D = \lfloor (2c+d)/3\rfloor$ and 
\begin{equation}
\label{eq:nt}
n_t = \left\{ 
\begin{array}{ll}
t, & 1 \le t < (c+d)/2\\[0.2cm]
(2c+d)-3t, & (c+d)/2 < t \le D.
\end{array} 
\right.
\end{equation}
In addition, if $c+d=2t^*$ is even (which can happen as seen in Example \ref{ex:6g}), then 
\begin{equation}
\label{eq:nt*}
6n_{t^*} = n - 1 - 6\sum_{t \ne t^*} n_t = n-1-6D(2c+d)+9D(D+1)+3(c-1)(c+d).
\end{equation}
It is known that, for any first-kind Frobenius graph, we have $\pi = 2\overrightarrow{\pi} = 2\overrightarrow{\pi}_m = \pi_m = 2\sum_{i=0}^D t n_t$ and these achieve the trivial lower bounds in (\ref{eq:4ineq}) (see \cite[Theorem 1.6]{FLP} and \cite[Theorem 6.1]{Z}).  
Using this and (\ref{eq:nt})--(\ref{eq:nt*}), we can give an explicit formula for $\pi(\Ga) = 2\overrightarrow{\pi}(\Ga) = 2\overrightarrow{\pi}_m(\Ga) = \pi_m(\Ga)$. If $c+d$ is odd, this quantity is equal to
\begin{equation}
\label{eq:odd}
D(D+1)[(2c+d)-(2D+1)] - \frac{1}{12}(2c-d)[(c+d)^2-1];
\end{equation}
if $c+d$ is even, it is equal to 
\begin{equation}
\label{eq:even}
\frac{1}{2}D(D+1)[(7c+5d)-2(2D+1)] + \frac{1}{12}(c+d)^2(4c+d-6)+\frac{1}{6}[n-3c-5-6D(2c+d)].
\end{equation}
 
It is known \cite[Theorem 5.1]{Z} that the minimum gossip time of any first-kind Frobenius graph achieves the trivial lower bound (\ref{eq:t}). This yields, for $\Gamma$ in (\ref{eq:Gamma}), 
\begin{equation}
\label{eq:t1}
t(\Ga) = (n-1)/6.
\end{equation}  
In particular, for HARTS $H_k$ (see Example \ref{optimal}) we get $t(H_k) = k(k-1)/2$, which was first proved in \cite[Theorem 4]{TZ1}. (In \cite[\S 4.3]{ABF} a gossiping algorithm for $H_k$ using $3k(k-1)/2$ time steps was devised.) In Algorithm \ref{alg:go} we will give an optimal gossiping algorithm for any 6-valent first-kind Frobenius circulant. 

\medskip
\textsf{C. Geometric representation.}~~In \cite{Z} a general method for producing optimal gossiping and routing schemes in a first-kind Frobenius graph was described. This method is abstract and relies on knowledge of the orbits of the complement on the kernel of the underlying Frobenius group. For 6-valent first-kind Frobenius circulants, we are able to acquire such knowledge with the help of a geometric representation \cite{YFMA}. 

We label the cells of the hexagonal lattice \cite{YFMA} in the plane by 
$$
\Z_+ \times \Z_+ \times \Z_6 \rightarrow \Z_n,\; (i, j, k) \mapsto [(i+ja)a^k],
$$
where $\Z_+$ is the set of nonnegative integers. (See Figure \ref{fig:hexagon3}, where, for example, $(3, 1, 0) \mapsto [(3+31)31^0] = [34]$ and $(2, 1, 2) \mapsto [(2+31)31^2] = [10]$ as $n=49$ and $a=31$.) The distance in $\Ga$ between $[u] \in \Z_n$ and $[0]$ is then given by
\begin{equation}
\label{eq:dist}
d\left([0],[u]\right)=\min\left\{i+j: \exists (i, j, k) \in \Z_+ \times \Z_+ \times \Z_6,\;u \equiv (i+ja)a^k\;\mbox{\rm mod}\; n\right\}.
\end{equation} 

Let $C_{\ell}$ be the set of hexagonal cells distant $\ell$ apart from a fixed $[0]$-labelled cell in the hexagonal lattice, $\ell = 1, 2, \ldots$ Then $C_{\ell}$ consists of those cells with coordinates $(i, \ell-i, k) \in \Z_+ \times \Z_+ \times \Z_6$, $1 \le i \le \ell$, $0 \le k \le 5$. 
The $H$-orbit on $\Z_n$ containing $[x] \in \Z_n$ is $H[x] = \{[a^i x]: i \ge 0\} = \{[a^i x]: 0 \le i \le 5\}$. In order to describe our optimal gossiping and routing schemes, we construct a `minimum distance diagram' $X$ by using the following algorithm.  

\begin{algorithm}
\label{alg:x}
\begin{itemize}
\item[\rm 1.] To begin with we put the six elements of $H[1]$ into $X$.  

\item[\rm 2.] Set $\ell := 2$ and do the following:
\begin{itemize}
\item[\rm (a)] Examine the cells $(\ell, 0, 0), (\ell-1, 1, 0), \ldots, (1, \ell-1, 0)$ of $C_{\ell}$ one by one {\em in this order}. When examining $(i, \ell-i, 0)$, if $H[i+(\ell-i)a]$ is not contained in the current $X$, add all its elements to $X$ and then move on to examine the next cell $(i-1, \ell-i+1, 0)$; otherwise examine the next cell straightaway. 

\item[\rm (b)] Set $\ell:=\ell + 1$ and go to Step 2(a). 

\item[\rm (c)] Stop when all elements of $\Z_n \setminus \{[0]\}$ are contained in $X$. 
\end{itemize} 
\end{itemize} 
\end{algorithm}

In the final $X$ each element of $\Z_n \setminus \{[0]\}$ appears exactly once, and $X$ tessellates the plane \cite{YFMA}. See Figure \ref{fig:hexagon3} for $TL_{49}(31, 30, 1)$. 

\begin{figure}[ht]
\centering
\includegraphics*[height=6.0cm]{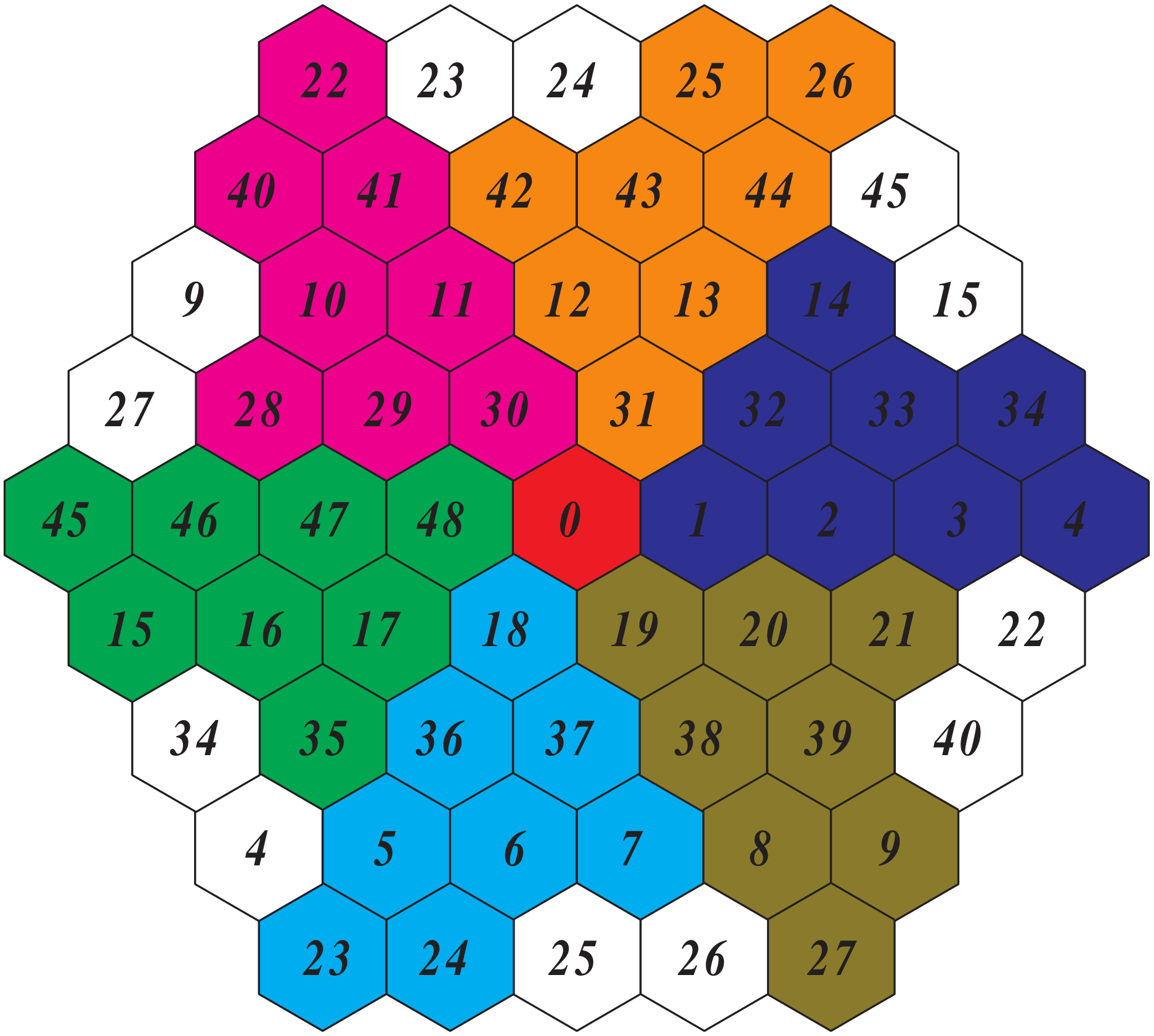}
\caption{\small Hexagonal tessellation of $TL_{49}(31, 30, 1)$. The coloured area is 
the minimum distance diagram $X \cup \{[0]\}$, where $Y=\{[1], [2], [3], [4], [32], [33], [34], [14]\}$ is the part of $X$ in the first sector. The other five sectors $Y[31], Y[30], -Y, -Y[31], -Y[30]$ of $X$ are obtained by rotating $Y$ about the origin by $60^{\circ}, 120^{\circ}, 180^{\circ}, 240^{\circ}, 300^{\circ}$ respectively. This graph has diameter $4$ and type $(i_0, i_1, i_2) = (4, 3, 1)$.} 
\label{fig:hexagon3}
\end{figure}  

Slightly abusing terminology, we may take $X$ as the set of cells $(i, \ell-i, k)$ such that $[(i+(\ell-i)a)a^k] \in X$. The shape of $X$ is determined by the values of the parameter $i_j$ defined as follows. Let
$$
r = \max\left\{i \ge 1: \mbox{$(i, 0, 0)$ is contained in $X$}\right\}. 
$$ 
Then $d([0], [i]) = i$, $0 \le i \le r$, for otherwise $r=d([0], [r]) \le d([0], [i])+d([i], [r]) = d([0], [i])+d([0], [r-i]) \le (i-1)+(r-i)$. On the other hand, for any $i \ge r+1$, $d([0], [i]) = d([0], [ia]) \le i-1$ by the definition of $r$. Thus, if $j \ge r+1$ and $i \ge 1$, then $(i, j, 0)$ is not contained in $X$ for otherwise $i+j = d([0], [i+ja]) \le d([0], [ja]) + d([ja], [i+ja]) = d([0], [ja]) + d([0], [i]) \le (j-1)+i$.
Define
$$
i_j = \max\left\{i \ge 0: \mbox{$(i, j, 0)$ is contained in $X$}\right\},\;0 \le j \le r. 
$$ 
Then $i_0 = r$ and $i_j$ is well-defined as $d([0], [ja]) = d([0], [j]) = j$ and so $(0, j, 0)$ belongs to $X$. The values of $i_j$ can be obtained by running Algorithm \ref{alg:x}.

Denote by $Y$ the subset of $X$ in the first sector of the hexagonal lattice.  

\begin{lem}
\label{lem:tel}
With the notation above, the following hold:
\begin{itemize}
\item[\rm (a)] $Y = \{[i+ja]: 1 \le i \le i_j, 0 \le j \le r\}$, $X = \cup_{k=0}^5 Y [a^k] = \{[(i+ja)a^k]: 1 \le i \le i_j, 0 \le j \le r, 0 \le k \le 5\}$, and every element of $\Z_{n} \setminus \{[0]\}$ appears in $X$ exactly once;
\item[\rm (b)]
if $[i+ja] \in Y$, then $d\left([0], [(i+ja)a^k]\right) = i+j$, $0 \le k \le 5$; 
\item[\rm (c)] $\sum_{j=0}^{r}i_j = (n-1)/6 \ge r = i_0 \ge i_1 \ge \cdots \ge i_r \ge 0$;
\item[\rm (d)] $D = \max\{i_j+j: 0 \le j \le r\}$; 
\item[\rm (e)] $n_t = |\Ga_{t}[0] \cap Y| = |\{[i+ja] \in Y: i+j =t\}| = |\{j: 0 \le j \le r, i_j + j \ge t\}|$, $1 \le t \le D$.
\end{itemize}
\end{lem}
\begin{proof}
(a) It suffices to prove that, if $i_j \ge 1$ for some $0 \le j \le r$, then $d([0], [i+ja]) = i+j$ for every $1 \le i \le i_j$. Suppose otherwise. Then $i_j + j = d([0], [i_j + ja]) \le d([0], [i+ja]) + d([i+ja], [i_j + ja]) = d([0], [i+ja]) + d([0], [i_j - i]) \le (i+j-1) + (i_j - i) = i_j + j - 1$, a contradiction.

(c) Suppose $i_{j-1} < i_{j}$ for some $j$. Then $(i_{j-1}+1)+(j-1) \ge d([0], [(i_{j-1}+1)+(j-1)a]) \ge d([0], [i_{j}+ja]) - d([(i_{j-1}+1)+(j-1)a], [i_{j}+ja]) = d([0], [i_{j}+ja]) - d([0], [(i_j - i_{j-1} - 1)+a]) \ge (i_j + j) - (i_j - i_{j-1}) = (i_{j-1}+1)+(j-1)$. Thus $d([0], [(i_{j-1}+1)+(j-1)a]) = (i_{j-1}+1)+(j-1)$, which contradicts the definition of $i_{j-1}$. So we have $i_{j-1} \ge i_{j}$ for $1 \le j \le r$. The truth of $\sum_{j=0}^{r}i_j = (n-1)/6$ follows from (a) and the symmetry of $X$.

The truth of (b), (d) and (e) follows from (a) and the definition of $X$. 
\qed \end{proof}
 
Part (a) of Lemma \ref{lem:tel} implies that $X$ is partitioned into six sectors, namely $Y, Y[a], Y[a^2] = Y[a-1], Y[a^3] = -Y, Y[a^4] = -Y[a], Y[a^5] = -Y[a-1]$, which are permuted cyclically by $H$. 

\medskip  
\textsf{D. Optimal routing and gossiping schemes.}~~Guided by the general approach in \cite{Z}, we now construct a spanning tree $T_0$ of $\Ga$ rooted at $[0]$ and use it to give optimal gossiping and routing in $\Ga$. Let 
$A_{1,1} = \{([0], [v]): [v] \in H\}$
and add these six arcs to $T_0$. Inductively, for $0 \le t \le D-1$ and each 
$[v_l] \in \Ga_{t+1}[0] \cap Y$ ($1 \le l \le n_{t+1}$), choose a neighbour $[u_l]$ of $[v_l]$ in $\Ga_{t}[0] \cap Y$ and add arcs 
$A_{t+1, l} = \{([u_la^k], [v_la^k]): 0 \le k \le 5\}$
to $T_0$. (It is allowed to have $u_l = u_{l'}$ for $l \ne l'$.) Thus the branches of $T_0$ in $Y[a^k]$ are obtained by rotating the branch of $T_0$ in $Y$ by $(60k)^{\rm o}$ and the set of arcs of $T_0$ from $T_0(t)$ to $T_0(t+1)$ is $\cup_{1 \le l \le n_{t+1}}A_{t+1, l}$, where $T_0(t)$ is the set of vertices distant $t$ apart from $[0]$ in $T_0$. Since $H$ is semiregular on $\Z_n \setminus \{[0]\}$, one can show that each $A_{t+1, l}$ is a matching of six arcs (see \cite{Z}). Note that $T_0(t) = \Ga_t[0]$, $0 \le t \le D$, and $T_0$ is a {\em shortest path spanning tree} of $\Ga$ with root $[0]$, that is, the unique path in $T_0$ between $[0]$ and any vertex is a shortest path in $\Ga$. 

For $[u] \in \Z_n$, define $T_u$ to be the graph with vertex set $\Z_n$ and arcs $([x+u], [y+u])$ with $([x], [y])$ running over all arcs of $T_0$. Since $\Z_n$ acts on itself (by addition) as a group of automorphisms of $\Ga$, $T_u$ is a shortest path spanning tree of $\Ga$ with root $[u]$. Denote by $P_{uv}$ the unique path in $T_u$ from $[u]$ to $[v]$. Define 
\begin{equation}
\label{eq:P}
{\cal P} = \{P_{uv}: [u], [v] \in \Z_n, [u] \ne [v]\}.
\end{equation}
  
\begin{algorithm}
\label{alg:go} ~Let $M_u$ denote the message originating at $[u] \in \Z_n$.

{\sc Phase 1:}~Initially, $M_u$ is transmitted from $[u]$ to $T_{0}(1)+[u]$ along the six arcs of $A_{1,1}+[u]$, and this is carried out for all $[u] \in \Z_{n}$ simultaneously.

{\sc Phase $t+1$:}~Do the following for $t=1, 2, \ldots,D-1$ successively: for $l = 1, 2, \ldots, n_{t+1}$, in the $l$th step of the $(t+1)$th phase, for all $[u] \in \Z_{n}$ transmit $M_u$ from $T_{0}(t)+[u]$ to $T_{0}(t+1)+[u]$ along the six arcs of $A_{t+1,l}+[u]$ at the same time step.
\end{algorithm}

A routing $\cal P$ of $\Ga$ is called {\em $G$-arc-transitive} \cite{LP} for some $G \le \Aut\Ga$ if every element of $G$ maps paths of $\cal P$ to paths of $\cal P$ and moreover $G$ is transitive on the set of arcs of $\Ga$. A routing under which all edges (arcs, respectively) have the same load is called {\em edge-uniform} ({\em arc-uniform}, respectively). 
The following is a consequence of Theorem \ref{existence}, Lemma \ref{lem:tel}, \cite[Theorem 1.6]{FLP} and \cite[Theorems 5.1 and 6.1]{Z}. 

\begin{cor}
\label{cor:rt-gs}
Let $\Ga = TL_{n}(a, a-1, 1)$ be a 6-valent first-kind Frobenius circulant, where each prime factor of $n \ge 7$ is congruent to 1 modulo 6 and $a$ is a solution to (\ref{eq:3rd}). Then
$\pi(\Ga) = 2\overrightarrow{\pi}(\Ga) = 2\overrightarrow{\pi}_m(\Ga) = \pi_m(\Ga)$ and it is given by (\ref{eq:odd}) or (\ref{eq:even}) (depending on whether the corresponding $c+d$ is odd or even), and $t(\Ga)$ is given by (\ref{eq:t1}).
Moreover, $\cal P$ given in (\ref{eq:P}) is a shortest path routing of $\Ga$ which is $\Z_{n} \rtimes H$-arc transitive (where $H$ is as given in (\ref{eq:H})), edge- and arc-uniform, and optimal for $\pi$, $\overrightarrow{\pi}$, $\overrightarrow{\pi}_m$ and $\pi_m$ simultaneously.

Furthermore, Algorithm \ref{alg:go} gives an optimal gossiping scheme for $\Ga$ such that: (a) the message originating from any vertex is transmitted along shortest paths to other vertices; (b) for each vertex $[w]$ of $\Ga$, at any time precisely six arcs are used to transmit the message originating from $[w]$, and at any time $\ge 2$ these six arcs form a matching of $\Ga$; (c) at any time each arc of $\Ga$ is used exactly once for message transmission.
\end{cor}

We remark that the spanning tree $T_0$ constructed above is not unique, and different choices of $T_0$ produce different optimal gossiping and routing schemes.

Since $\Ga$ achieves the trivial lower bounds in (\ref{eq:4ineq}), by Lemma \ref{lem:tel} we obtain 
\begin{equation}
\label{eq:fw}
\pi(\Ga) = \sum_{j=0}^r i_j (i_j + 2j + 1).
\end{equation}

\begin{ex}
\label{ex:mix}
It can be verified that, for $\Ga = TL_{n_k} (3k+2, 3k+1, 1)$ ($k \ge 2$) in Example \ref{optimal}, we have $i_j = k-j$ ($0 \le j \le k-1$) and $i_k = 0$. From this and (\ref{eq:fw}) we recover the result $\pi(\Ga) = k(k+1)(2k+1)/2$ obtained in \cite[Theorem 5]{TZ1}. In \cite{TZ1}, the authors also gave optimal routing and gossiping schemes for this particular graph. Corollary \ref{cor:rt-gs} generalizes these to all 6-valent first-kind Frobenius circulants. 

The graph $\Ga$ in Example \ref{ex:6g} satisfies $i_j = 2g-j$ ($0 \le j \le g-1$), $i_j = 2g-j-1$ ($g \le j \le 2g-1$) and $i_{2g} = 0$. From this we obtain $\pi(\Ga) = 2g(8g^2+1)/3$ by (\ref{eq:fw}) and $t(\Ga) = 2g^2$ by (\ref{eq:t1}).
\qed \end{ex}

\textsf{E. Wiener index.}~~The {\em Wiener index} of a graph is the sum of the distances between all unordered pairs of vertices. With motivation from chemistry, this index has attracted considerable interest in chemical graph theory over sixty years (see \cite{DGKP} for a survey on the topic for hexagonal systems). As a by-product of the discussion above, we obtain the following result.

\begin{cor}
\label{cor:w}
The Wiener index of any 6-valent first-kind Frobenius circulant $TL_{n}(a, a-1, 1)$ is equal to $3n/2$ times the expression in (\ref{eq:odd}) or (\ref{eq:even}), depending on whether the corresponding $c+d$ is odd or even, or equivalently $3n/2$ times the right-hand side of (\ref{eq:fw}).
\end{cor}

\section{Broadcasting}
\label{sec:br}

A process of disseminating a message from a {\em source vertex} $x$ to all other vertices in a network $\Ga$ is called {\em broadcasting} \cite{HKMP} (one-to-all communication) if in each time step any vertex who has received the message already can retransmit it to at most one of its neighbours. Let $b(\Ga, x)$ be the minimum $t$ such that all vertices receive the message after $t$ steps. The {\em broadcasting time} \cite{HKMP} of $\Ga$, denoted by $b(\Ga)$, is the maximum among $b(\Ga, x)$ for $x$ running over all vertices of $\Ga$. 
 
Since the diameter is a trivial lower bound on the broadcasting time, any graph whose broadcasting time is close to its diameter may be thought as efficient in terms of broadcasting. The following result shows that all 6-valent first-kind Frobenius circulants are such graphs. 
 
\begin{thm}
\label{thm:broad}
Let $\Ga = TL_{n}(a, a-1, 1)$ be a 6-valent first-kind Frobenius circulant, where each prime factor of $n \ge 7$ is congruent to 1 modulo 6 and $a$ is a solution to (\ref{eq:3rd}). Let $D$ be the diameter of $\Ga$. Then 
\begin{equation}
\label{eq:bounds}
b(\Ga) = D+2\;\, \mbox{or}\;\, D+3
\end{equation}
and both $D+2$ and $D+3$ are attainable.
In particular, if $n = 12g^2+1 \ge 49$ and $a = 6g^2 + 3g + 1$ as in Example \ref{ex:6g}, then 
\begin{equation}
\label{eq:ub}
b(\Ga) = D+3 = 2g+3.
\end{equation}
\end{thm}

\begin{proof}
We use the notation and results in the previous section. In particular, $X$ denotes the `minimum distance diagram' of $\Ga$ and $Y$ the first sector of it. Since $\Ga$ is vertex-transitive, it suffices to prove $D+2 \le b(\Ga, [0]) \le D+3$.  

We prove the lower bound first. Let $M$ denote the message at $[0]$ to be broadcasted to other vertices. At time 1 the message is sent from $[0]$ to exactly one of the six vertices of $H$. At time 2 the message can be sent to at most two vertices of $H$. So at least three vertices of $H$ receive $M$ at time 3 or later, and if there are exactly three such vertices then two of them must be `consecutive' around $[0]$. 
One can verify that in any case at least one vertex whose distance to $[0]$ is equal to $D$ receives $M$ at time $D+2$ or later. (Note that at least six vertices are at distance $D$ apart from $[0]$.) Therefore, $b(\Ga, [0]) \ge D+2$.
  
We prove the upper bound by giving a broadcasting scheme explicitly. Before doing so let us explain our notation first. A broadcasting scheme with source vertex $[0]$ can be defined by specifying a pair $L(x) = (t_x, y_x)$ for each $x \ne [0]$, which means that $x$ receives the message at time $t_x$ from a neighbour $y_x$ of $x$. We require $t_y < t_x$ for $y=y_x$ and $(t_x, y_x) \ne (t_z, y_z)$ if $x \ne z$.

Using the notation above, we define
$$
L([1]) = (1, [0]), L([a]) = (2, [1]), L([a^3]) = (2, [0])
$$
$$ 
L([a^2]) = (3, [a]), L([a^4]) = (3, [a^3]), L([a^5]) = (3, [0]).
$$
For each $k =0, 1, \ldots, 5$, define
$$
L([ia^k]) = (i+2, [(i-1)a^k]),\; \mbox{for}\; 2 \le i \le r
$$
$$
L([(i+ja)a^k]) = (i+j+3, [(i+(j-1)a)a^k]),\; \mbox{for}\; 1 \le i \le r-1,\, j \ge 1,\, [i+ja] \in Y
$$
$$
L([(r+ja)a^k]) = (r+j+2, [(r+(j-1)a)a^k]),\; \mbox{for}\; j \ge 1,\, [r+ja] \in Y.
$$
It is straightforward (but laborious) to verify that this $L$ defines a broadcasting scheme for $\Ga$. Since the maximum value of $i+j$ such that $[i+ja] \in Y$ is equal to $D$, this broadcasting can be completed in at most $D+3$ time steps. Therefore, $b(\Ga, [0]) \le D+3$. Moreover, if there is only one vertex $[u]$ in $Y$ such that $d([0], [u]) = D$ and further $[u]$ is of the form $[r + ja]$ for some $j \ge 0$, then $L$ requires only $D+2$ time steps and hence $b(\Ga) = D+2$. This occurs when, for example, $\Ga = TL_{43}(7, 6, 1) \cong EJ_{7-6\rho}$. Thus the lower bound in (\ref{eq:bounds}) is attainable. 

We now prove (\ref{eq:ub}) for $\Ga$ with $n = 12g^2+1 \ge 49$ and $a = 6g^2 + 3g + 1$ (see Example \ref{ex:6g}). As mentioned in Example \ref{ex:mix}, this special graph $\Ga$ satisfies $i_j = 2g-j$ ($0 \le j \le g-1$), $i_j = 2g-j-1$ ($g \le j \le 2g-1$) and $i_{2g} = 0$. Hence its diameter $D = 2g$. It suffices to prove $b(\Ga, [0]) \ge 2g+3$. 

Suppose to the contrary that there exists a broadcasting scheme for $\Ga$ using $2g+2$ time steps. By Lemma \ref{lem:tel}(a), we have $d([0], [((2g-j)+ja)a^k]) = 2g$ for $0 \le j \le g-1$ and $0 \le k \le 5$. In particular, $d([0], [2ga^k]) = 2g$ and $P_k: [0], [a^k], [2a^k], \ldots, [2ga^k]$ is the unique shortest path from $[0]$ to $[2ga^k]$. Among the six vertices of $H$, exactly one receives $M$ at time 1. If at most one vertex of $H$ receives $M$ at time 2, then at least one vertex in $H$, say, $[a^k]$, receives $M$ at time 4. Since $P_k$ has length $2g$, this implies that $[2ga^k]$ receives $M$ at time $2g+3$ or later, which contradicts our assumption. Thus exactly two vertices of $H$ receive $M$ at time 2, and the remaining three vertices receive $M$ at time 3 or later. However, if a vertex in $H$, say, $[1]$, receives $M$ at time 4 or later, then since $P_0$ is the unique path from $[0]$ to $[2g]$, $[2g]$ receives $M$ at time $2g+3$ or later, which is a contradiction. Therefore the times that the vertices of $H$ (in cyclic order) receive $M$ must be $(1, 2, 2, 3, 3, 3)$, $(1, 2, 3, 2, 3, 3)$, $(1, 2, 3, 3, 2, 3)$ or $(1, 2, 3, 3, 3, 2)$. In each case there are two consecutive vertices of $H$ (with respect to the cyclic order) which receive $M$ at time 3. Without loss of generality we may assume that $[1]$ and $[a]$ receive $M$ at time $3$. Since the broadcasting finishes in $2g+2$ time steps, and since $P_0$ is the unique shortest path from $[0]$ to $[2g]$ and its length is $2g$, each vertex $[i]$ on $P_0$ has to receive $M$ from $[i-1]$ at time $i+2$, $2 \le i \le 2g$. Similarly, each vertex $[ia]$ on $P_1$ receives $M$ from $[(i-1)a]$ at time $i+2$, $2 \le i \le 2g$. Since $d([0], [(2g-1)+a]) = d$, $[(2g-1)+a]$ has to receive $M$ via a shortest path $Q$ from $[0]$ to $[2g-1]$. Note that $Q$ uses either $[1]$ or $[a]$ as its second vertex. In the former case, $Q$ is of the form $[0], [1], \ldots, [s], [s+a], [s+a+1], \ldots, [(2g-1)+a]$ for some $1 \le s \le 2g-1$. However, since $[s]$ sends $M$ to $[s+1]$ at time $s+3$, it cannot send $M$ to $[s+a]$ at time $s+3$. Thus $[s+a]$ receives $M$ from $[s]$ at time $s+4$ or later. Consequently, $[s+a+1]$ receives $M$ from $[s+a]$ at time $s+5$ or later, and so on. Finally, $[(2g-1)+a]$ receives $M$ from $[(2g-2)+a]$ at time $2g+3$ or later, contradicting our assumption. In the case when $Q$ uses $[a]$, it is the path $[0], [a], [a+1], \ldots, [a+(2g-1)]$. Since $[a]$ sends $M$ to $[2a]$ at time $4$, it cannot send $M$ to $[a+1]$ at time 4. Hence $[a+1]$ receives $M$ from $[a]$ at time $5$ or later, and so on, and $[a+(2g-1)]$ receives $M$ at time $2g+3$ or later, which is again a contradiction. Therefore, $b(\Ga [0]) \ge 2g+3$ and (\ref{eq:ub}) holds for the graph in Example \ref{ex:6g}.
\qed \end{proof}

We notice that for HARTS $H_k$ (see Example \ref{optimal}) an optimal broadcasting algorithm was given in \cite[Algorithm A2]{CSK}, which uses $k+2$ steps if $k \ge 3$ and $k+1 = 3$ steps if $k = 2$. Since $H_k$ is a 6-valent first-kind Frobenius circulant of diameter $k-1$ (Example \ref{optimal}), this result is now a special case of Theorem \ref{thm:broad}.

\section{Concluding remarks}

It would be interesting to find whether non-Frobenius EJ graphs are as efficient as 6-valent first-kind Frobenius circulants in terms of routing, gossiping and broadcasting, and whether they also have the smallest possible forwarding indices and gossip time. 

In Example \ref{optimal} we saw that, for any integer $k \ge 2$, $TL_{n_k} (3k+2, 3k+1, 1)$ is a 6-valent first-kind Frobenius circulant with $n_k = 3k^2+3k+1$ vertices. When $k=2$ this graph is the HARTS network physically tested at the University of Michigan as a distributed real-time computing system. Note that $n_k$ can be a composite number (e.g.~$n_5 = 91 = 7 \cdot 13$), and in this case by (e) of Theorem \ref{existence}, besides $TL_{n_k} (3k+2, 3k+1, 1)$ there is at least one more 6-valent first-kind Frobenius circulant of order $n_k$. It would be interesting to explicitly construct all of them.

Finally, one may investigate various combinatorial properties of 6-valent first-kind Frobenius circulants and EJ graphs in general. The results in \cite{MNR} imply that the chromatic number of any 6-valent first-kind Frobenius circulant $\Ga$ of order $n \equiv 1~\mod 6$ is equal to $4$ if $n \ne 7, 13, 19$, and $7, 5, 5$ respectively in these exceptional cases. Thus, if $n > 19$, then the independence number of $\Ga$ is at least $\lceil n/4 \rceil$. At present we do not know whether this bound is sharp in general.

\bigskip
\noindent {\bf Acknowledgements}~~We thank the anonymous referees for bringing \cite{ABF, Shin} to our attention and appreciate Alex Ghitza for helpful discussions on Eisenstein-Jacobi integers. Zhou was supported by a Future Fellowship (FT110100629) and a Discovery Project Grant (DP120101081) of the Australian Research Council, as well as a Shanghai Leading Academic Discipline Project (No. S30104).

\end{document}